\documentclass[noams]{compositio}

\usepackage[latin1]{inputenc}
\usepackage[T1]{fontenc}
\usepackage[francais]{babel}
\usepackage[mathscr]{eucal}
\usepackage{epic,eepic}
\usepackage{graphics}

\usepackage{amsmath}
\usepackage{amssymb}
\usepackage{tabularx}
\usepackage{epsfig, floatflt,amssymb}

\usepackage[all]{xy}

\newcommand{\PP}{\mathrm{P}}
\newcommand{\HH}{\mathrm{H}}

\newcommand{\OO}{\mathcal{O}}

\newcommand{\QQ}{\mathcal{Q}}
\newcommand{\R}{\mathrm{R}}

\newcommand{\Gr}{\mathrm{Gr}}

\newcommand{\Sch}{\mathbf{Sch}}

\newcommand{\Pic}{\mathrm{Pic}}

\newtheorem{theorem}[subsubsection]{Théorème}
\newtheorem{prop}[subsubsection]{Proposition}
\newtheorem{rem}[subsubsection]{Remarque}
\newtheorem{lemma}[subsubsection]{Lemme}
\newtheorem{corollary}[subsubsection]{Corollaire}
\newtheorem{exam}[subsubsection]{Exemple}
\newtheorem{definition}[subsubsection]{Définition}

\newtheorem{theoreme}{Théorème}[subsection]

\newtheorem{remarques}[theoreme]{Remarque}

\newenvironment{remark}{\begin{rem} \upshape}{\end{rem}}

\renewcommand{\contentsname}{Contents\\{\footnotesize\normalfont(A table
of contents should normally not be included)}}

\begin{document}

\title{Etude locale des torseurs sous une courbe elliptique}
\author{Jilong Tong}
\email{jilong.tong@uni-due.de}
\address{Universität Duisburg-Essen, Fachbereich Mathematik, Campus Essen 45117 Essen, Germany}



\classification{14H99.} \keywords{torseur sous une courbe
elliptique, fonctions de Herbrand, foncteur de Picard, réalisation
de Greenberg}

\begin{abstract}This article concerns the geometry of torsors
under an elliptic curve. Let $\OO_K$ be a complete discrete
valuation ring with algebraically closed residue field and
function field $K$. Let $\pi$ be a generator of the maximal ideal
of $\OO_K$, and $S=\mathrm{Spec}(\OO_K)$. Suppose that we are
given $J_K$ an elliptic curve over $K$, with $J$ the connected
component of the $S$-Néron model of $J_K$. Given $X_K/K$ a torsor
of order $d$ under $J_K$, let $X$ be the $S$-minimal regular
proper model. Then there is an invertible idéal
$\mathcal{I}\subset \OO_K$ such that
$\mathcal{I}^{d}=\pi\OO_X\subset \OO_X$. Moreover, there exists a
canonical morphism $q:\Pic^{\circ}_{X/S}\rightarrow J$ which
induces a surjective map $q(S):\Pic^{\circ}(X)\rightarrow J(S)$.
The purpose of the article is to prove this last morphism $q(S)$
is compatible with respect to the $\mathcal{I}$-adic filtration on
$\Pic^{\circ}(X)$, and the $\pi$-adic filtration on $J(S)$. As a
byproduct, we obtain {\textquotedblleft Herbrand
functions\textquotedblright}, similar to those Serre used in his
description of local class fields (\cite{Serre}).

\end{abstract}

\maketitle


\section*{Introduction} Soit $\OO_K$ un anneau de valuation
discrète complet, à corps résiduel $k$ algébriquement clos de
caractéristique $p>0$, à corps des fractions $K$, et soit $\pi\in
\OO_K$ une uniformisante de $\OO_K$. Notons
$S=\mathrm{Spec}(\OO_K)$, avec $s$ le point fermé. Soit $J_{K}$
une courbe elliptique sur $K$, et soit $\mathcal{N}$ son
$S$-modèle de Néron, $J=\mathcal{N}^{\circ}$ sa composante neutre.
Donnons-nous par ailleurs un torseur $X_{K}$ sous $J_{K}$ d'ordre
$d$, et soit $X$ le $S$-modèle propre minimal régulier de $X_{K}$.
En général, $X$ n'est pas cohomologiquement plat (en degré $0$)
sur $S$ (\emph{i.e.}, le morphisme canonique $k\rightarrow
\HH^{0}(X_{s},\mathcal{O}_s)$ n'est pas un isomorphisme), en
particulier, son foncteur de Picard $\Pic^{\circ}_{X/S}$ n'est pas
représentable, même par un espace algébrique. On montre dans
\cite{Raynaud} qu'il existe un épimorphisme (pour la topologie
fppf) de foncteurs en groupes naturel
$q:\Pic^{\circ}_{X/S}\rightarrow J$ qui prolonge l'isomorphisme de
bidualité sur la fibre générique. De plus, le pgcd des
multiplicités des composantes irréductibles de $X_s$ est $d$
(\ref{d1=d2=d}), il existe donc un faisceau inversible d'idéaux
$\mathcal{I}$ de $\OO_X$ tel que $\mathcal{I}^{d}=\pi\OO_X\subset
\OO_X$. Le but de cet article est d'étudier les faisceaux
inversibles sur $X$ en relation avec la filtration
$\mathcal{I}$-adique, et ensuite de montrer que le morphisme $q$
ci-dessus est compatible avec la filtration $\mathcal{I}$-adique
sur $\Pic^{\circ}_{X/S}(S)$, et la filtration $\pi$-adique sur
$J(S)$. Tout ceci se dit agréablement sur les réalisations de
Greenberg de $\Pic^{\circ}_{X/S}$ et $J$. Cette étude conduit
aussi à des fonctions de Herbrand, analogues à celles rencontrées
par Serre (\cite{Serre}) dans la description du corps de classes
local.

\section{Rappels sur le foncteur de Picard et le foncteur de Greenberg.}\label{recall on Picard}


On rassemble ici des résultats bien connus concernant le foncteur
de Picard. De manière générale, pour $X$ un schéma, on note
$(\Sch/X)$ la catégorie des $X$-schémas.

\subsection{Rappels sur le foncteur de Picard.}\label{Rappels Pic}

\subsubsection{}\label{assumption de depart} Soit $f:X\rightarrow S$ un
morphisme propre, notons
$$
\mathrm{Pic}_{X/S}: (\Sch/S)\rightarrow \mathfrak{Ab}
$$
le foncteur de Picard de $X/S$, c'est-à-dire, le faisceau fppf
associé au préfaiseau $S'\mapsto \mathrm{Pic}(X\times_S S')$.
C'est aussi le faisceau étale associé au préfaisceau $S'\mapsto
\Pic(X\times_S S')$ (\cite{Raynaud} 1.2).

\subsubsection{} \label{plat}Supposons jusqu'à la fin de la section
$\S$\ref{Rappels Pic} que $f$ est propre et plat. En général, le
foncteur $\mathrm{Pic}_{X/S}$ n'est pas représentable (même par un
espace algébrique), et il est représentable par un $S$-espace
algébrique si et seulement si $X/S$ est \emph{cohomologiquement
plat (en degré $0$)}, i.e., si la formation de
$f_{\ast}\mathcal{\mathcal{O}}_X$ commute aux changements de base
quelconques $S'\rightarrow S$. Bien que le foncteur
$\mathrm{Pic}_{X/S}$ n'est pas représentable, il possède une jolie
présentation par des $S$-espaces algébriques. Pour cela, il nous
faut d'abord rappeler la notion de rigidificateur.

\begin{definition}[\cite{BLR} 8.1/5] \label{rigidificateur}Soit
$i:Y\hookrightarrow X$ un $S$-sous schéma fermé avec $Y$ fini plat
sur $S$. On dit que $(Y,i)$ est un \emph{rigidificateur} de
$\mathrm{Pic}_{X/S}$ si la condition suivante est remplie: pour
tout $S$-schéma $S'$, si $i':Y'\rightarrow X'$ désigne le
morphisme déduit de $i$ par le changement de base $S'\rightarrow
S$, l'application
$$
\Gamma(i'):\Gamma(X',\mathcal{O}_{X'})\rightarrow
\Gamma(Y',\mathcal{O}_{Y'})
$$
est injective.
\end{definition}

\subsubsection{} Sous l'hypothèse de \ref{assumption de depart} et
de \ref{plat}, le faisceau $\mathrm{Pic}_{X/S}$ possède toujours
un rigidificateur (\cite{Raynaud} proposition 2.2.3 (c)).
Supposons donné $(Y,i)$ un rigidificateur de $\mathrm{Pic}_{X/S}$,
et pour tout $S$-schéma $S'$, on appelle un \emph{faisceau
inversible sur $X'=X\times_S S'$, rigidifié le long du
rigidificateur $Y'$}, un couple $(\mathcal{L}, \alpha)$, où
$\mathcal{\mathcal{L}}$ est un faiscau inversible sur $X'$, et
$\alpha: \mathcal{O}_{Y'}\simeq i'^{\ast}\mathcal{L}$ est un
isomorphisme (autrement dit, $\alpha$ est une trivialisation de
$i'^{\ast}\mathcal{L}$). Un isomorphisme entre deux faisceaux
inversibles rigidifiés $(\mathcal{L},\alpha)$ et
$(\mathcal{M},\beta)$ sur $X'$ est la donnée d'un isomorphisme de
$\mathcal{O}_{X'}$-modules $u:\mathcal{L}\rightarrow \mathcal{M}$
tel que le diagramme suivant soit commutatif:
$$
\xymatrix{i'^{\ast}\mathcal{L}\ar[rr]^{i'^{\ast}u} & &
i'^{\ast}\mathcal{M} \\ &
\mathcal{O}_{Y'}\ar[lu]^{\alpha}\ar[ru]_{\beta} & }.
$$

\subsubsection{}\label{def de Q} On note $(\mathrm{Pic}_{X/S},Y)(S')$ l'ensemble
des classes d'isomorphisme de faisceaux inversibles sur $X'$,
rigidifiés le long de $Y'$. Pour $S'$ variable dans la catégorie
des $S$-schémas $\mathbf{Sch}/S$, l'application $S'\mapsto
(\mathrm{Pic}_{X/S},Y)(S')$ définit un foncteur en groupes
commutatifs $(\mathrm{Pic}_{X/S},Y)$, appelé le \emph{le foncteur
de Picard de $X/S$ relatif au rigidificateur $Y$}. Concernant la
représentabilité de $(\mathrm{Pic}_{X/S},Y)$, on a

\begin{theorem}[\cite{Raynaud} théorème 2.3.1 et corollaire 2.3.2] Le foncteur $(\mathrm{Pic}_{X/S},Y)$ est représentable par
un $S$-espace algébrique en groupes, localement de présentation
finie sur $S$. De plus, si $X/S$ est une courbe, le $S$-espace
algébrique $(\Pic_{X/S},Y)$ est lisse sur $S$.
\end{theorem}

\subsubsection{}\label{presentation de P} On a un morphisme de faisceaux en groupes
$r:(\mathrm{Pic}_{X/S},Y)\rightarrow \mathrm{Pic}_{X/S}$, qui au
couple $(\mathcal{L},\alpha)$ associe l'image de $\mathcal{L}$
dans $\mathrm{Pic}_{X/S}(S')$. Étale localement, tout élément de
$\mathrm{Pic}_{X/S}(S')$ peut être représenté par un faiseau
inversible sur $X'$ (\ref{assumption de depart}), le morphisme $r$
est donc un épimorphisme pour la topologie étale. Pour étudier son
noyau, notons $V_{X}^{\ast}$ (resp. $V_{Y}^{\ast}$) le faisceau
abélien fppf sur $(\mathbf{Sch}/S)$, donné par $S'\mapsto
\Gamma(X', \OO_{X'})^{\ast}$ (resp. $S'\mapsto
\Gamma(Y',\OO_{Y'})^{\ast})$. Par définition de rigidificateur
(\ref{rigidificateur}), le morphisme naturel
$V_{X}^{\ast}\rightarrow V_{Y}^{\ast}$ est injectif. Posons $u$ le
morphisme défini de la façon suivante (pour $S'$ un $S$-schéma):
$$
u:V_{Y}^{\ast}\rightarrow (\Pic_{X/S},Y), ~~~~a\in
V_{Y}^{\ast}(S')=\Gamma(Y_{S'},\OO_{Y_{S'}}^{\ast})\mapsto
(\OO_{X_{S'}},\alpha_a)\in (\Pic_{X/S},Y)(S')
$$
où $\alpha_{a}:\OO_{Y_{S'}}\rightarrow
\OO_{Y_{S'}}=\OO_{X_{S'}|_{Y_{S'}}}$ est la multiplication par
$a$. Clairement, $\mathrm{im}(u)\subset \ker(r)$. On obtient ainsi
un complexe de faisceux fppf sur $S$:
$$
\xymatrix{0\ar[r] & V_{X}^{\ast}\ar[r] & V_{Y}^{\ast}\ar[r]^{u} &
\left(\mathrm{Pic}_{X/S},Y\right)\ar[r]^{r}&
\mathrm{Pic}_{X/S}\ar[r]& 0},
$$
qui est exact pour la topologie étale (\cite{Raynaud} 2.1.2(b) et
2.4.1). De plus, le morphisme $r$ est formellement lisse au sens
de Grothendieck (\cite{EGAIV} EGA IV, 17.1.1).  Rappelons que les
deux premiers membres du complexe ci-dessus sont représentables
par $S$-schémas, et le schéma $X/S$ est cohomologiquement plat (en
degré $0$) si et seulement si $V_{X}^{\ast}$ est un $S$-schéma
plat, et en fait lisse.




\subsection{Modèle de Néron et foncteurs de Picard.}

\subsubsection{} Soit $f:X\rightarrow S$ une courbe propre et plate
 (à fibres géométriques connexes).
Notons $\mathrm{P}$ (resp. $(\mathrm{P},R)\emph{}$) le
sous-foncteur (ouvert) de $\mathrm{Pic}_{X/S}$ (resp.
$(\mathrm{Pic}_{X/S},R)$) consistant en les faisceaux inversibles
de degré total $0$ (resp. les faisceaux inversibles rigidifiés le
long $R$ de degré total $0$). Alors, $(\mathrm{P},R)$ est un
sous-espace algébrique ouvert de $(\mathrm{Pic}_{X/S},R)$, et
$\mathrm{P}$ (resp. $(\mathrm{P},R)$) est l'adhérence schématique
de $\left(\mathrm{Pic}_{X/S}\right)_{K}^{\circ}$ (resp.
$(\mathrm{Pic}_{X/S},R)_{K}^{\circ}$) dans $\mathrm{Pic}_{X/S}$
(resp. dans $(\mathrm{Pic}_{X/S},R)$). Notons $E$ l'adhérence
schématique de l'élément neutre de $\mathrm{P}_{K}$ dans
$\mathrm{P}$, et définissons $\mathcal{Q}$ comme le quotient fppf
de $\mathrm{P}$ par $E$. C'est le plus grand quotient séparé de
$\mathrm{P}$. Il est représentable par un schéma en groupes séparé
et lisse sur $S$ (\cite{Raynaud} théorème 3.3.1). On désigne par
$q$ le morphisme canonique $\PP\rightarrow \mathcal{Q}$, qui est
donc surjectif pour la topologie fppf.

\begin{theorem}[\cite{LLR} 3.7]\label{Neron} Gardons les
notations ci-dessus, et supposons de plus $X$ régulier, et
$f_{\ast}\mathcal{O}_{X}=\mathcal{O}_S$. Alors, le schéma en
groups $\mathcal{Q}/S$ est le modèle de Néron de
$\mathrm{P}_{K}=\mathrm{Pic}^{\circ}_{X_{K}/K}$.
\end{theorem}

\subsubsection{} \label{Brauer nul}Notons $J=\mathcal{Q}^{\circ}$
la composante neutre de $\QQ$. Comme $\OO_K$ est complet à corps
résiduel algébriquement clos, d'après un résultat de S. Lang
(théorème 1.1, \cite{Grothendieck}), le groupe de Brauer de $K$
est nul. Il en résulte que tout élément de $\Pic_{X/S}(K)$ peut se
représenter par un faisceau inversible sur $X_K$. Par suite, le
morphisme naturel $q(S):\Pic^{\circ}_{X/S}(S)\rightarrow J(S)$ est
surjectif (9.6 de \cite{BLR}).


\subsection{Foncteur de Greenberg et foncteur de Picard} \label{Greenberg}

\subsubsection{}\label{Greenberg lisse} Soit $G$ un schéma en groupes
lisse de type fini sur $S$. Les foncteurs de Greenberg nous permet
de construire un k-groupe pro-algébrique ($:=$ un pro-objet dans
la catégorie des $k$-groupes algébriques). Rappelons d'abord
brièvement cette construction. Notons $W$ l'anneau de Witt du
corps $k$, et $\mathbf{W}$ le foncteur de Witt sur la catégorie
des $k$-algèbres $\mathbf{Alg}/k$. Soit $n\in \mathbf{Z}_{\geq
1}$, notons $\OO_{K,n}=\OO_{K}/\pi^{n}$. Alors $\OO_{K,n}$ est
canoniquement un $W$-module de longueur finie. On définit
$\mathrm{Gr}_n(G)$ comme le faisceau fpqc sur $\mathbf{Alg}/k$
associé au foncteur $A\mapsto G(\OO_{K,n}\otimes_W
\mathbf{W}(A))$. D'après Greenberg (\cite{Greenberg}), ce faisceau
est représentable par un $k$-schéma en groupes lisse de type fini.
Pour chaque entier $n\geq 1$, le morphisme canonique d'anneaux
$\OO_{K,n+1}\rightarrow \OO_{K,n}$ induit un morphisme lisse de
k-schémas en groupes $\alpha_{n}:\Gr_{n+1}(G)\rightarrow
\Gr_n(G)$, dont le noyau est un groupe unipotent connexe sur $k$.
De plus, le morphisme de groupes canonique
$G(\OO_{K,n})\rightarrow \mathrm{Gr}_n(G)(k)$ est un isomorphisme.
Sous cette identification, le morphisme
$\alpha(k):\Gr_{n+1}(G)(k)\rightarrow \Gr_n(G)(k)$ s'identifie
naturellement au morphisme canonique $G(\OO_{K,n+1})\rightarrow
G(\OO_{K,n})$. Les $k$-groupes algébriques $\Gr_n(G)$ forment un
système projectif $\{(\Gr_n(G),\alpha_n)\}_{n\geq 1}$ de la
catégorie des $k$-groupes algébriques, dont les morphismes de
transition sont tous lisses à noyau connexe.

\subsubsection{}\label{Greenberg Pic} Soit $f:X\rightarrow S$ une courbe propre, telle
que $f(X)=\{s\}$. Alors si $X \!\neq\! \emptyset$, $X/S$ n'est
jamais plat, et les foncteurs de Picard $\Pic_{X/S}$ et
$\Pic^{\circ}_{X/S}$ ($:=$ le sous-foncteur ouvert de $\Pic_{X/S}$
formé des faisceaux inversibles de degré $0$ sur chaque composante
irréductible de $X$) ne sont pas représentables. Néanmoins, comme
l'a montré Lipman (\cite{Lipman}), la réalisation de Greenberg de
$\Pic_{X/S}$ (resp. de $\Pic^{\circ}_{X/S}$) est représentable par
un $k$-schéma en groupes lisse. Plus précisément, comme $\OO_K$
est naturellement une $W=W(k)$-algèbre, on trouve que $X$ possède
une structure de $W$-schéma propre. On pose $\Gr(\Pic_{X/S})$
(resp. $\Gr(\Pic^{\circ}_{X/S})$) le faisceau fpqc associé au
foncteur
$$
\mathbf{Alg}/k\rightarrow \mathfrak{Ab}, ~~~~~~ A\mapsto
\Pic(X\otimes_{W}W(A))~~~~ (resp. ~~A\mapsto
\Pic^{\circ}(X\otimes_{W}W(A))).
$$

\begin{theorem}[\cite{Lipman}]\label{Lipman} Le foncteur $\Gr(\Pic_{X/S})$
(resp. $\Gr(\Pic^{\circ}_{X/S})$) est représentable par un schéma
en groupes lisse sur $k$ (resp. par un schéma en groupes lisse
connexe sur $k$), qui est de dimension la longueur du $W$-module
$\HH^{1}(X,\OO_{X})$. De plus, le morphisme canonique
$$
\Pic(X)\rightarrow \Gr(\Pic_{X/S})(k), ~~~~(resp.
~~\Pic^{\circ}(X)\rightarrow \Gr(\Pic^{\circ}_{X/S})(k))
$$
est un isomorphisme.
\end{theorem}

\section{Fonctions de Herbrand}\label{Herbrand}

Gardons les notations précédentes. Soit $f:X\rightarrow S$ une
$S$-courbe propre plate, telle que les conditions suivantes soient
remplies: (i) $f_{\ast}(\OO_X)=\OO_S$; (ii) $X_{K}$ est une courbe
géométriquement intègre \emph{lisse} de genre arithmétique $1$
(\emph{i.e.} $h^{1}(X_K,\mathcal{O}_{X_K})=1$); (iii) $X$ est une
surface régulière minimale sur $S$. En particulier, $X_{K}$ est un
torseur sous sa jacobienne $J_{K}$. Soit
$X_s=\sum_{i=1}^{r}n_iC_i$ la décomposition de $X_s$ en
composantes irréductibles réduites, et notons $d$ le pgcd des
$n_i$. On désigne par $D$ le diviseur
$$
\frac{1}{d}X_s=\sum_{i=1}^{r}\frac{n_i}{d}\cdot C_i,
$$
et notons $\mathcal{I}$ l'idéal $\OO_{X}(-D)$. La fibre spéciale
$X_k$ est donc définie par $\mathcal{I}^{d}=\pi\OO_{X}\subset
\OO_X$. Pour tout $n\in \mathbf{Z}_{\geq 1}$, on note $X_n$ le
sous-schéma fermé de $X$ défini par l'idéal
$\mathcal{I}^{n}\subset \OO_X$. Le but de ce numéro est d'étudier
la variation de $n\mapsto h^{1}(X,\mathrm{O}_{X_n})$ ($:=$ la
longueur de $\OO_{K}$-module $\HH^{1}(X,\OO_{X_n})$).
\footnote{Une partie de résultats de cette section a été rédigée
dans \cite{KatsuraUeno}.}

\subsection{Etude du faisceau dualisant}\label{etude de dualisant}

\subsubsection{}\label{notationss}Commençons par un résultat
classique.

\begin{lemma}\label{d1=d2=d} Notons $d_1$ l'ordre du torseur $X_{K}$ dans
$\HH^{1}(\mathrm{Spec}(K),J_{K})$, $d_2$ le minimum des degrés des
extensions $K'$ de $K$ telles que $X(K')\neq \emptyset$, et $d_3$
le minimum des multiplicités des composantes irréductibles de
$X_s$. Alors $d_1=d_2=d_3=d$.
\end{lemma}

\begin{proof} Pour $n\in
\mathbf{Z}_{>0}$ un entier positif, le torseur $n\cdot X_{K}$ est
isomorphe à la composante irréductible $\Pic^{n}_{X/S}$ de
$\Pic_{X_K/K}$ formée des faisceaux inversibles de degré $n$. Par
suite, le torseur $n\cdot X_K$ est trivial si et seulement si
$\Pic^{n}_{X/K}(K)\neq \emptyset$. Or $\OO_K$ est strictement
hensélien à corps résiduel algébriquement clos, on a
$\mathrm{Br}(K)=0$ (\ref{Brauer nul}). Par suite,
$\Pic^{d_2}_{X/K}(K)=\Pic^{d_2}(X)\neq \emptyset$. Soit
$\Sigma_K\subset X_K$ un diviseur de degré $d_2$, et notons
$\Sigma$ son adhérence schématique dans $X$. Alors
$d_2=\Sigma\cdot X_s=d\cdot \Sigma\cdot D$ est divisible par $d$.
Clairement, on a $d_1|d_2$. Par ailleurs, comme $\OO_K$ est
strictement hensélien, pour chaque $i$, on peut trouver un
diviseur positif (de Cartier relatif) $\Delta_i$ de $X/S$ de degré
$n_i$ (\cite{BLR}). Une combinaison convenable nous fournit un
diviseur $\Delta'$ de degré $d$ de $X_K$. Or $X_K$ est une courbe
de genre $1$, et $d\geq 1$, on en déduit que
$h^{0}(X_K,\OO_{X_{K}}(\Delta_K'))>0$. Il existe donc un diviseur
positif $\Delta_K$ de degré $d$ de $X_K$ linéairement équivalent à
$\Delta_K'$. Par suite $\Delta_K$ est intègre, \emph{i.e.},
$\Delta_K=\{x\}$ avec $x\in X_K$ un point fermé de degré $d$. D'où
$d_2\leq d$, et on a donc $d=d_1=d_2$. Soit maintenant $x\in X_K$
un point fermé de degré $d$, notons $\Delta=\overline{\{x\}}$
l'adhérence schématique de $\{x\}$ dans $X$. Alors $\Delta\cdot
X_s=d(\Delta\cdot D)=d$. Donc $\Delta\cap D=\{y\}$, et $D$ est
régulier en $y$. Notons $D_i$ la composante irréductible de $D$
telle que $y\in D$, alors $D$ est de multiplicité $d$ dans $X_s$,
d'où $d=d_3$. Ceci termine la démonstration.
\end{proof}

\begin{remark} Soit $Y\subset X$ un diviseur effectif plat de degré
$d$ sur $S$ (dont l'existence est assurée par \ref{d1=d2=d}).
Notons que $Y$ est nécessairement intègre régulier, et coupe
transversalement une unique composante $C_i$, de multiplicité $1$
dans $D$. Alors $Y\hookrightarrow X$ est un rigidificateur pour le
foncteur de Picard $\Pic_{X/S}$. En fait, d'après 2.2.2 de
\cite{Raynaud}, il suffit de vérifier l'injectivité du morphisme
canonique $\HH^{0}(X_s, \OO_{X_s})\rightarrow
\HH^{0}(Y_s,\OO_{Y_s})$. Montrons par récurrence sur $n$ que le
morphisme canonique $\HH^{0}(X_n,\OO_{X_n})\rightarrow
\HH^{0}(Y_n,\OO_{Y_n})$ est injectif (où $Y_n:=Y\cap X_n$).
Commençons par le cas où $n=1$: d'après \ref{lemme cohomologique}
ci-après, on sait que $\HH^{0}(X_1,\OO_{X_1})=k$. Soit
$\varepsilon\in \HH^{0}(X_1,\OO_{X_1})$, alors $\varepsilon$ est
une fonction globale de $X_1$ qui est constante. Par conséquent,
l'image de $\varepsilon$ dans $\HH^{0}(Y_1,\OO_{Y_1})$ est nulle
si et seulement si $\varepsilon=0$, autrement dit, le morphisme
$\HH^{0}(X_1,\OO_{X_1})\rightarrow \HH^{0}(Y_{1},\OO_{Y_1})$ est
injectif. Supposons ensuite l'assertion ci-dessus vérifiée pour
$n=n_0\geq 1$. Partons du diagramme commutatif à lignes exactes
suivant:
$$
\xymatrix{0\ar[r]& \mathcal{I}^{n_0}|_{X_1}\ar[r]\ar[d]&
\OO_{X_{n_0+1}}\ar[r]\ar[d]& \OO_{X_{n_0}}\ar[r]\ar[d]& 0
\\0\ar[r]& \mathcal{I}^{n_0}|_{Y_1}\ar[r]&
\OO_{Y_{n_0+1}}\ar[r]& \OO_{Y_{n_0}}\ar[r]& 0 }.
$$
On voit qu'il suffit de vérifier l'injectivité du morphisme
$\HH^{0}(X_1,\mathcal{I}^{n_0}|_{X_1})\rightarrow
\HH^{0}(Y_1,\mathcal{I}^{n_0}|_{Y_1})$. En vertu du lemme
\ref{lemme cohomologique}, on peut supposer
$\mathcal{I}^{n_0}|_{X_1}\simeq \OO_{X_1}$, auquel cas on peut
identifier le morphisme de gauche au morphisme canonique
$\HH^{0}(X_1,\OO_{X_1})\rightarrow \HH^{1}(Y_1,\OO_{Y_1})$ qui est
injectif d'après ce qui précède. D'où l'assertion.
\end{remark}

\subsubsection{} On note $\omega_{X/S}=f^{!}\mathcal{O}_{S}$ le faisceau dualisant
relatif sur $X/S$. Pour tout $n\geq 0$, notons $\omega_n$ le
faisceau dualisant sur $X_n$. Donc
$\omega_{n}=\left(\OO_{X}(nD)\otimes \omega_{X/S}\right)|_{X_n}$.

\begin{lemma} Pour tout $i=1,\cdots,r$, on a $\omega_{X/S}\cdot
C_i=0$.
\end{lemma}
\begin{proof}  Comme $\omega_{X/S}|_{X_{\eta}}\simeq
\OO_{X_\eta}$, on a $\omega_{X/S}\cdot X_{s}=0$, \emph{i.e.},
$\sum_{i=1}^{r}n_i\left(\omega_{X/S}\cdot C_i\right)=0$. En
particulier, si $r=1$, le lemme en résulte. Supposons $r\geq 2$,
puisque $C_i\cdot X_k=0$, on obtient que $C_i\cdot C_i<0$. Si
$\omega_{X/S}\cdot C_i<0$, puisque
$2g(C_i)-2=(\omega_{X/S}+C_i)\cdot C_i\geq -2$, on a donc
$g(C_i)=0$, $C_i\cdot C_i=-1$. Ceci contredit le fait que $X/S$
est une surface régulière minimale. Donc $(\omega_{X/S}\cdot
C_i)\geq 0$, il en résulte que $\omega_{X/S}\cdot C_i=0$ pour tout
$i$.
\end{proof}

\begin{corollary}\label{dualisant} Il existe un unique entier $n$, $0\leq n< d$, tel
que $\omega_{X/C}\simeq \mathcal{I}^{n}$.
\end{corollary}

\begin{proof} Puisque $\omega_{X/S}|_{X_{\eta}}\simeq \OO_{X_{\eta}}$,
$\omega_{X/S}\simeq \OO_{X}(Y)$, avec $Y$ un diviseur de $X$ à
support dans $X_k$. Par suite, $Y$ est une combinaison des $C_i$.
Or d'après le lemme précédent, $Y\cdot C_i=0$, on obtient $Y\cdot
Y=0$, par suite, $Y$ est un multiple rationnel de $X_k$,
c'est-à-dire, $Y$ est linéairement équivalent à $nD$ avec $0\leq
n<d$. Le corollaire s'en déduit.
\end{proof}

\begin{corollary}\label{Classique} Supposons que $f$ possède une section $s$,
définie par le faiceau d'idéaux $\mathcal{J}$, et soit
$\omega=\mathcal{J}/\mathcal{J}^{2}$, on a un isomorphisme
canonique $\omega_{X/C}\simeq f^{\ast}\omega$.
\end{corollary}

\begin{proof} Par l'hypothèse, le torseur $X_{K}$ possède un point
rationnel, il est donc trivial en tant que torseur sous $J_K$. En
vertu du lemme \ref{d1=d2=d}, on a donc $d=1$. Par conséquent,
$\omega_{X/C}\simeq \OO_{X}$ (corollaire \ref{dualisant}), et le
morphisme canonique $f^{\ast}f_{\ast}\omega_{X/C}\rightarrow
\omega_{X/C}$ est un isomorphisme. Par ailleurs, on a des
isomorphismes canoniques: $\OO_S\simeq (fs)^{!}\OO_{S}\simeq
s^{!}(\omega_{X/S})[-1]\simeq s^{\ast}\omega_{X/S}\otimes
\omega^{\vee}$. Donc $\omega\simeq s^{\ast}\omega_{X/S}\simeq
s^{\ast}f^{\ast}f_{\ast}\omega_{X/S}\simeq f_{\ast}\omega_{X/S}$,
d'où l'isomorphisme canonique $f^{\ast}\omega\simeq
f^{\ast}f_{\ast}(\omega_{X/S})\simeq \omega_{X/S}$.
\end{proof}

\subsubsection{} \label{X'} Dans le cas général où $f$ n'a pas nécessairement de
section, on peut considérer le $S$-modèle propre régulier minimal
$f':X'\rightarrow S$ de $X_{K}'=\Pic^{\circ}_{X_K/K}$. Par suite,
$f'$ possède une section canonique $e$ ($=$ l'adhérence
schématique de l'élément neutre de $X'_{K}=J_{K}$ dans $X'$). Son
faisceau dualisant est $f{'}^{\ast}\omega$ (avec $\omega$ défini
par la section $e$ de $X'/S$, voir \ref{Classique}). On peut
retrouver $\omega_{X/S}$, à partir du faisceau $\omega$, en
utilisant certains invariants numériques de $X/S$. On renvoie à la
section \ref{Etude numerique} pour plus de détails.

\subsubsection{} Le lemme suivant est très utile dans la suite. On
trouvera une preuve dans \cite{Mumford}.

\begin{lemma}[\cite{Mumford}, page 332]\label{lemme cohomologique}
Soit $L$ un faisceau inversible sur $X_1$, de degré $0$ sur chaque
composante. Alors, si $\HH^{0}(X_1,L)\neq 0$, on a $L\simeq
\OO_{X_1}$, et $\HH^{0}(X_1,\OO_{X_1})\simeq k$.
\end{lemma}

\subsubsection{}\label{exact1} Soient $n\geq 2$ un entier, et $L$ un faisceau
inversible sur $X$, de degré $0$ sur chaque composante de $X_1$.
Considérons la suite exacte suivante:
\begin{equation}\label{suite de depart}
0\rightarrow \OO_{X}(-D)|_{(n-1)D}\rightarrow \OO_{nD}\rightarrow
\OO_D\rightarrow 0,
\end{equation}
en tensorisant par
$L^{\vee}\otimes \omega_{X/S}(nD)$, on obtient une suite exacte:
$$
0\rightarrow L^{\vee}\otimes\omega_{n-1}\rightarrow
L^{\vee}\otimes\omega_{n}\rightarrow
L^{\vee}\otimes\omega_{n}|_{D}\rightarrow 0.
$$
D'où une suite exacte
\begin{equation}\label{exact}
0\rightarrow
\HH^{0}(X_{n-1},L^{\vee}\otimes\omega_{n-1})\rightarrow
\HH^{0}(X_n,L^{\vee}\otimes\omega_n)\rightarrow
\HH^{0}(X_1,L^{\vee}\otimes\omega_n|_{D}).
\end{equation}
Par conséquent, on a
$$
h^{0}(X_{n-1},L^{\vee}\otimes\omega_{n-1})\leq
h^{0}(X_n,L^{\vee}\otimes\omega_n)\leq
h^{0}(X_{n-1},L^{\vee}\otimes\omega_{n-1})+1.
$$

\begin{lemma} \label{suite exacte de dualisant}Gardons les notations
ci-dessus. Alors, ou bien $\omega_{n}\simeq L|_{X_n}$, auquel cas
$\HH^{0}(X_1,L^{\vee}\otimes\omega_n|_{D})\simeq k$ et le complexe
(\ref{exact}) est exact à droite; ou bien $\omega_n\ncong
L|_{X_n}$, auquel cas le morphisme canonique
$$
\xymatrix{\HH^{0}(X_{n-1},L^{\vee}\otimes
\omega_{n-1})\ar[r]^{\simeq}&
\HH^{0}(X_n,L^{\vee}\otimes\omega_n)}
$$
est bijectif.
\end{lemma}

\begin{proof} Supposons d'abord $\omega_n\simeq L|_{X_n}$. Alors
la suite exacte (\ref{exact}) se réécrit sous la forme suivante:
\begin{equation}\label{suite1}
0\rightarrow \HH^{0}(X_{n-1},\OO_{X}(-D)|_{(n-1)D})\rightarrow
\HH^{0}(X_n,\OO_{X_n})\rightarrow \HH^{0}(X_1,\OO_{X_1}),
\end{equation}
qui est aussi la suite exacte longue déduite de la suite exacte
courte (\ref{suite de depart}). D'après le lemme \ref{lemme
cohomologique}, on a $ \HH^{0}(X_1,\OO_{X_1})\simeq k $, i.e., les
fonctions globales de $X_1$ sont les fonctions constantes sur
$X_1$. Donc tout élément de $\HH^{0}(X_1,\OO_{X_1})$ peut se
relever en un élément de $\HH^{0}(X_n,\OO_{X_n})$. Par conséquent,
le complexe (\ref{suite1}) (et donc le complexe (\ref{exact})) est
exact à droite, d'où la première assertion. Ensuite, supposons
$\omega_n\ncong L|_{X_n}$, auquel cas, même si $(\omega_n\otimes
L^{\vee})|_{X_{1}}\simeq \OO_{X_1}$, une section non nulle de
$(\omega_n\otimes L^{\vee})|_{X_{1}}$ ne peut jamais se relever en
une section de $\omega_n\otimes L^{\vee}|_{X_n}$, (sinon
$\omega_n\otimes L^{\vee}|_{X_n}$ serait trivial). Donc le
morphisme canonique
$$
\xymatrix{\HH^{0}(X_{n-1},L^{\vee}\otimes
\omega_{n-1})\ar[r]^{\simeq}&
\HH^{0}(X_n,L^{\vee}\otimes\omega_n)}
$$
est bijectif.
\end{proof}

\subsubsection{}\label{notations} Soit $n\geq 2$ un entier,
alors le noyau du morphisme surjectif $\Pic(X_n)\rightarrow
\Pic(X_{n-1})$ est un $\OO_K$-module de longueur finie annulé par
$p$. Plus précisément, considérons l'immersion fermée
$X_{n-1}\hookrightarrow X_n$, son idéal de définition est l'idéal
cohérent $\mathfrak{N}:=\mathcal{I}^{n-1}/\mathcal{I}^{n}\subset
\OO_{X_{n}}$. Le faisceau $\mathfrak{N}$ est nilpotent (en fait,
$\mathfrak{N}^{2}=0$), on a donc une suite exacte:
$$
0\rightarrow 1+\mathfrak{N}\rightarrow \OO_{X_n}^{\ast}\rightarrow
\OO_{X_{n-1}}^{\ast}\rightarrow 0.
$$
Comme $X_n$ est de dimension $1$, la cohomologie
$\HH^{2}(X_n,1+\mathfrak{N})\simeq \HH^{2}(X_n,\mathfrak{N})$ est
nulle. D'où une suite exacte longue:
$$
\xymatrix{\HH^{0}(X_{n-1},\OO_{X_{n-1}}^{\ast})\ar[r]^{\partial^{\ast}}&
\HH^{1}(X_{n},1+\mathfrak{N})\ar[r]& \Pic(X_n)\ar[r]^{\alpha}&
\Pic(X_{n-1})\ar[r]& 0}. ~~(\ast)
$$
D'autre part, à partir de la suite exacte suivante:
$$
0\rightarrow \mathfrak{N}\rightarrow \OO_{X_n}\rightarrow
\OO_{X_{n-1}}\rightarrow 0,
$$
on obtient une suite exacte longue (rappelons que
$\HH^{2}(X_n,\mathfrak{N})=0$):
$$
\xymatrix{\HH^{0}(X_{n-1},\OO_{X_{n-1}})\ar[r]^{\partial}&
\HH^{1}(X_{n},\mathfrak{N})\ar[r]& \HH^{1}(X_{n},\OO_{X_n})
\ar[r]^{\alpha'}& \HH^{1}(X_{n-1},\OO_{X_{n-1}})\ar[r]& 0}.
$$
Par ailleurs, comme $\mathfrak{N}^{2}=0$, le morphisme $x\mapsto
1+x$ définit un isomorphisme de faisceaux abéliens
$$
\beta: \mathfrak{N}\rightarrow 1+\mathfrak{N}.
$$
On a alors le résultat suivant:
\begin{lemma}[Dévissage d'Oort, \cite{Oort} $\S$ 6 proposition]\label{Oort}
Gardons les notations ci-dessus, alors
$\beta(\mathrm{im}(\partial))=\mathrm{im}(\partial^{\ast})$.
\end{lemma}

\subsubsection{}\label{def de phi} Par conséquent, $\ker(\alpha)\simeq \mathrm{coker}(\partial)$
(comme faisceaux abéliens). Comme
$\mathfrak{N}=\mathcal{I}^{n-1}/\mathcal{I}^{n}$ est un
$\OO_K$-module annulé par $p$, on en déduit que
$\mathrm{ker}(\alpha)\simeq \mathrm{coker}(\partial)$ est un
$\OO_K$-module de longueur finie tué par $p$. Ceci étant, notons
$d'$ l'ordre du faisceau inversible $\mathcal{I}|_{X_1}$, alors
pour $n\geq 2$, l'ordre de $\mathcal{I}|_{X_n}$ est de la forme
$d'p^{\ell}$. Par ailleurs, puisque $\mathcal{I}|_{X_K}\simeq
\OO_{X_K}$, d'après le lemme 6.4.4 de \cite{Raynaud}, pour $m\in
\mathbf{Z}$ assez grand, le faisceau inversible
$\mathcal{I}|_{X_m}$ est d'ordre $d$. Donc, $d=d'p^{r}$ avec
$r\geq 0$ un entier convenable. Pour $i=0,\cdots, r$, soit $m_i$
le plus petit entier $n$ tel que $\mathcal{I}|_{X_{n}}$ soit
d'ordre $d'p^{i}$. On pose aussi $\phi(n)=h^{1}(X,\OO_{X_n})$ la
longueur du $\OO_K$-module $\HH^{1}(X,\OO_{X_n})$. D'après
\ref{suite exacte de dualisant}, on a $\phi(n)\geq \phi(n-1)$. De
plus, $\phi(n)>\phi(n-1)$ si et seulement si $\omega_n\simeq
\OO_{X_n}$, auquel cas $\phi(n)=\phi(n-1)+1$. On déduit du lemme
\ref{Oort} le corollaire suivant:

\begin{corollary}\label{Pic_S} Soit $n\geq 2$ un entier. Alors ou bien
$\phi(n)=\phi(n-1)$, auquel cas le morphisme
$\alpha:\Pic(X_{n})\rightarrow \Pic(X_{n-1})$ est un isomorphisme;
ou bien $\phi(n)=\phi(n-1)+1$, auquel cas $\ker(\alpha)$ est un
$\OO_{K}$-module de longueur $1$, et donc un $k$-espace vectoriel
de dimension $1$.
\end{corollary}

\begin{lemma}\label{ki} Gardons les notations de \ref{def de phi}.
Alors:

(i) Pour $i=0,\cdots,r$, $\omega_{m_i}$ est trivial.

(ii) Il existe un entier $k_i>0$ tel que $m_{i+1}=m_i+k_id'p^{i}$.

(iii) Les entier $n\in [m_i,m_{i+1}]$, pour lesquels on a
$\phi(n)=\phi(n-1)+1$ sont ceux de la forme $n=m_i+hd'p^{i}$ avec
$h$ entier.
\end{lemma}

\begin{proof} Soit $n>1$ un entier tel que l'ordre de $\mathcal{I}|_{X_n}$
soit différent de celui de $\mathcal{I}_{X_{n-1}}$, alors le
morphisme canonique $\Pic(X_n)\rightarrow \Pic(X_{n-1})$ n'est pas
un isomorphisme. Par suite, $\phi(n)=\phi(n-1)+1$, et le morphisme
$\HH^{1}(X_n,\OO_{X_n})\rightarrow \HH^{1}(X_{n-1},\OO_{X_{n-1}})$
a un noyau de longueur $1$ (corollaire \ref{Pic_S}). Reprennons la
suite exacte (\ref{exact}) de \ref{exact1}. Par dualité, le
morphisme injectif $\HH^{0}(X_{n-1},\omega_{n-1})\rightarrow
\HH^{0}(X_n,\omega_n)$ a donc un conoyau non trivial. D'où
$\omega_n\simeq \OO_{X_n}$ (lemme \ref{suite exacte de
dualisant}). Pour (ii), rappelons que
$\omega_{m_i}=\omega_{X/S}(m_iD)|_{X_{m_i}}$, et d'après
\ref{dualisant}, il existe un entier $n$ ($1\leq n\leq d-1$) tel
que $\omega_{X/S}\simeq \mathcal{I}^{n}$. Notons
$\mathcal{L}_m:=\omega_{X/S}(mD)$, alors $\mathcal{L}_m\simeq
\mathcal{I}^{n-m}$. Or
$\omega_{m_{i+1}}=\mathcal{L}_{m_{i+1}}|_{X_{m_{i+1}}}=\mathcal{I}^{n-m_{i+1}}|_{X_{m_{i+1}}}\simeq
\OO_{X_{m_{i+1}}}$. Par suite
$\omega_{m_{i+1}}|_{X_{m_i}}=\mathcal{I}^{n-m_{i+1}}|_{X_{m_i}}\simeq
\OO_{X_{m_i}}$. Puisque
$\omega_{m_i}=\mathcal{I}^{n-m_i}|_{X_{m_i}}\simeq \OO_{X_{m_i}}$,
il en résulte que $\mathcal{I}^{m_i-m_{i+1}}|_{X_{m_i}}\simeq
\OO_{X_{m_i}}$, il existe donc un entier $k_i>0$ tel que
$m_{i+1}=m_i+k_id'p^{i}$. Le même raisonnement nous donne aussi
que, pour $m_{i+1}\geq m> m_i$ un entier tel que
$\phi(m)>\phi(m-1)$, il existe un entier $0<h\leq k_i$ vérifiant
$n=m_i+hd'p^{i}$. Réciproquement, soit $m$ un entier de la forme
$m=m_i+hd'p^{i}$ (avec $0<h\leq k_i$), prouvons que
$\phi(m)>\phi(m-1)$. On peut supposer que $m<m_{i+1}$, donc
$\mathcal{I}|_{X_m}$ est d'ordre $d'p^{i}$. Compte tenu de
\ref{suite exacte de dualisant}, il suffit de montrer que
$\omega_{m}\simeq \OO_{X_m}$. Or
\begin{eqnarray*}
\omega_{m}&\simeq &
\mathcal{I}^{n-m}|_{X_{m}}=\mathcal{I}^{n-m_i-hd'p^{i}}|_{X_{m}} \\
& =& \mathcal{I}^{n-m_{i+1}+m_{i+1}-n-hd'p^{i}}|_{X_{m_i}} \\
& \simeq & \omega_{m_{i+1}}|_{X_{m}}\otimes
\mathcal{I}^{(k_0-h)d'p^{i}}|_{X_m} \\ & \simeq &\OO_{X_m}
\end{eqnarray*}
car $\omega_{m_i}\simeq \OO_{X_{m_i}}$ et $\mathcal{I}|_{X_{m}}$
est d'ordre $d'p^{i}$. Ceci finit la démonstration.
\end{proof}

\subsubsection{} \label{def de psi} Définissons la fonction
$\varphi:\mathbf{R}_{\geq 0}\rightarrow \mathbf{R}_{\geq 0}$ telle
que son graphe soit l'enveloppe concave de l'ensemble
$\{(n,\phi(n))~|~n\in \mathbf{Z}_{\geq 0}\}\subset
\mathbf{R}^{2}$. Alors $\varphi$ est une fonction continue
strictement croissante, et linéaire par morceaux. De plus,
$\varphi(0)=0$, et $\varphi(1)=1$. Notons $\psi:\mathbf{R}_{\geq
0}\rightarrow \mathbf{R}_{\geq 0}$ son inverse. Donc, $\psi$ est
encore continue, et linéaire par morceaux. Pour tout entier $n\geq
1$, $\psi(n)$ est le plus petit entier $m\geq 1$ tel que
$\phi(m)=n$. Si l'on note $d_n$ l'ordre de faisceau inversible
$\mathcal{I}|_{X_{\psi(n)}}$, alors $\psi(n+1)=\psi(n)+d_n$ (lemme
\ref{ki}), et pour tout $m\in \mathbf{Z}$ tel que $\psi(n)\leq
m<\psi(n+1)$, le morphisme de groupes $ \Pic(X_{m})\rightarrow
\Pic(X_{\psi(n)})$ est un isomorpshime (corollaire \ref{Pic_S}).
Les fonctions $\varphi,\psi:\mathbf{Z}_{\geq 1}\rightarrow
\mathbf{Z}_{\geq 1}$ nous donnent des fonctions de Herbrand, tout
à fait similaires à celles de Serre dans sa description du corps
de classes local \cite{Serre}.

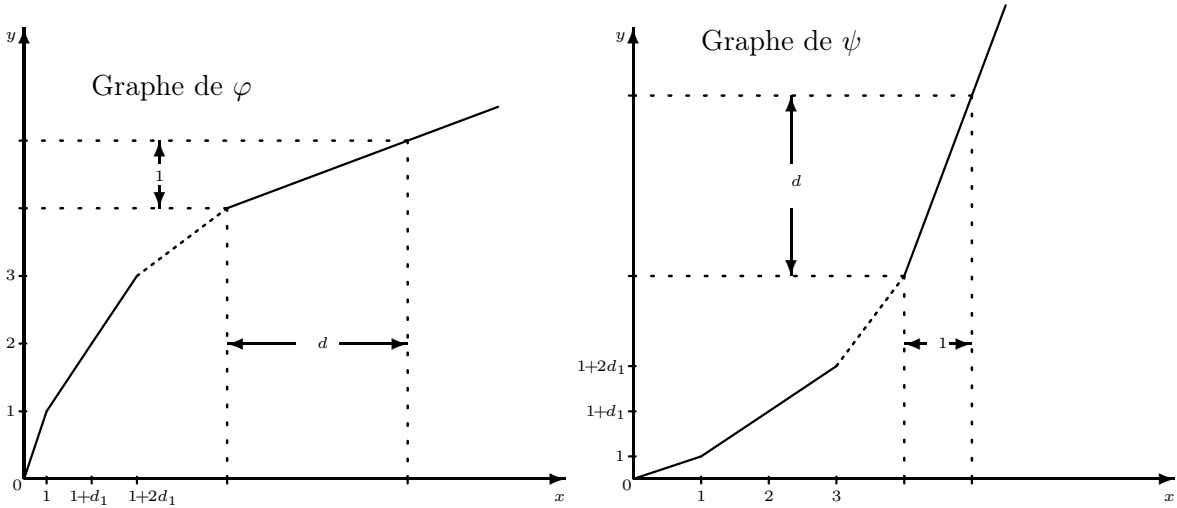
\setlength{\unitlength}{0.3cm}
\begin{picture}(70,25)
\thicklines \put(1,1){\vector(1,0){24}}
\put(1,1){\vector(0,1){20}} \drawline(1,1)(2,4)(4,7)(6,10)
\dottedline{0.4}(6,10)(10,13)\drawline(10,13)(18,16)(22,17.5)

\drawline(2,1.1)(2,0.8)\drawline(4,1.1)(4,0.8)\drawline(6,1.1)(6,0.8)
\drawline(10,1.1)(10,0.8)\drawline(18,1.1)(18,0.8)
\drawline(0.8,4)(1.1,4)\drawline(0.8,7)(1.1,7)\drawline(0.8,10)(1.1,10)
\drawline(0.8,13)(1.1,13)\drawline(0.8,16)(1.1,16)

\put(0.5,0.5){\tiny{$0$}} \put(1.8,0){\tiny{$1$}}
\put(3.0,0){\tiny{$1\!\!+\!\!d_1$}}
\put(5.6,0){\tiny{$1\!\!+\!\!2d_1$}}\put(24.5,0){\tiny{$x$}}

\dottedline{0.8}(10,13)(10,1) \dottedline{0.8}(18,16)(18,1)
\put(13,7){\vector(-1,0){3}}\put(15,7){\vector(1,0){3}}\put(14,6.8){\tiny{$d$}}

\put(0.2,3.8){\tiny{$1$}}\put(0.2,6.8){\tiny{$2$}}
\put(0.2,9.8){\tiny{$3$}}\put(0.2,20.5){\tiny{$y$}}

\dottedline{0.8}(1,13)(10,13)\dottedline{0.8}(1,16)(18,16)
\put(7,15){\vector(0,1){1}}\put(7,14){\vector(0,-1){1}}\put(6.8,14.2){\tiny{1}}

\put(4,18){Graphe de $\varphi$}



\put(28,1){\vector(1,0){24}} \put(28,1){\vector(0,1){20}}
\drawline(28,1)(31,2)(34,4)(37,6)
\dottedline{0.4}(37,6)(40,10)\drawline(40,10)(43,18)(44.5,22)

\drawline(31,1.1)(31,0.8)\drawline(34,1.1)(34,0.8)\drawline(37,1.1)(37,0.8)
\drawline(40,1.1)(40,0.8)\drawline(43,1.1)(43,0.8)

\drawline(27.8,2)(28.1,2)\drawline(27.8,4)(28.1,4)\drawline(27.8,6)(28.1,6)
\drawline(27.8,10)(28.1,10)\drawline(27.8,18)(28.1,18)

\put(27.5,0.5){\tiny{$0$}} \put(30.8,0){\tiny{$1$}}
\put(33.8,0){\tiny{$2$}} \put(36.8,0){\tiny{$3
$}}\put(51.5,0){\tiny{$x$}}

\dottedline{0.8}(40,10)(40,1) \dottedline{0.8}(43,18)(43,1)
\put(41,7){\vector(-1,0){1}}\put(42,7){\vector(1,0){1}}\put(41.5,6.8){\tiny{1}}

\put(27.2,1.8){\tiny{$1$}}\put(25.9,3.8){\tiny{$1\!\!+\!\!d_1$}}
\put(25.5,5.8){\tiny{$1\!\!+\!\!2d_1$}}\put(27.2,20.5){\tiny{$y$}}

\dottedline{0.8}(28,10)(40,10)\dottedline{0.8}(28,18)(43,18)
\put(35,13){\vector(0,-1){3}}\put(35,15){\vector(0,1){3}}\put(35,14){\tiny{$d$}}

\put(31,20){Graphe de $\psi$}

\end{picture}



\subsubsection{} On termine ce numéro avec certains
corollaires de \ref{suite exacte de dualisant}, qui nous seront
utiles dans la suite. Le premier corollaire résulte directement de
\ref{suite exacte de dualisant} par dualité.

\begin{corollary}\label{lemme cohomologie pour L} Soit $L$ un faisceau inversible sur $X$, de degré
$0$ sur chaque composante de $X_1$, et soit $n$ un entier $\geq
2$. Alors on a $h^{1}(X_{n-1},L)\leq h^{1}(X_n,L)\leq
h^{1}(X_{n-1},L)+1$. De plus, $ h^{1}(X_n,L)= h^{1}(X_{n-1},L)+1$
si et seulement si $L|_{X_n}\simeq \omega_n$
\end{corollary}

\begin{corollary}\label{condition pour L} Soit $L$ un faisceau inversible sur $X$, de degré
$0$ sur chaque composante de $X_1$, et soit $n$ un entier $\geq
1$. Alors si la flèche $\HH^{1}(X,L)\rightarrow \HH^{1}(X_n,L)$
n'est pas bijective, il existe un entier $m>n$ tel que
$L^{\vee}\otimes \omega_m$ soit trivial sur $X_m$.
\end{corollary}

\begin{proof} Notons que la flèche $\HH^{1}(X,L)\rightarrow
\HH^{1}(X_n,L|_{X_n})$ est surjective, et que
$\HH^{1}(X,L)=\varprojlim_{m\geq n}\HH^{1}(X_m, L|_{X_m})$. Pour
que la flèche ne soit pas bijective, il faut et il suffit qu'il
existe $m> n$ tel que $\HH^{1}(X_m, L|_{X_m})\rightarrow
\HH^{1}(X_{m-1},L|_{X_{m-1}})$ ne soit pas injectif. Par dualité,
ceci équivaut à dire que le morphisme injectif
$\HH^{0}(X_{m-1},L^{\vee}\otimes\omega_{m-1})\rightarrow
\HH^{0}(X_m, L^{\vee}\otimes\omega_{m})$ ne soit pas surjectif. On
a donc $L^{\vee}\otimes \omega_{m}\simeq \OO_{X_m}$ (\ref{etude de
dualisant}), d'où le résultat.
\end{proof}



\begin{corollary}\label{cle} Soit $L$ un faisceau inversible sur $X$
de degré $0$ sur chaque composante de $X_1$, et $n\geq 1$ un
entier. Supposons que le $\OO_K$-module $\HH^{1}(X,L)$ soit de
longueur $\geq n$, alors $L|_{X_{\psi(n)}}\simeq
\mathcal{I}^{i}|_{X_{\psi(n)}}$ avec $i$ un entier convenable.
\end{corollary}

\begin{proof} Montrons par récurrence que, sous l'hypothèse du corollaire, pour
tout $n'$ ($1\leq n'\leq n$), le $\OO_K$-module
$\HH^{1}(X_{\psi(n'+1)-1},L)$ est de longueur $n'$. Commençons par
le cas où $n'=1$. Regardons $\HH^{1}(X_1,L)$. D'après \ref{lemme
cohomologique}, ou bien ce $\OO_K$-module est de longueur $1$, ce
qui équivaut à dire que $L\simeq \OO_{X_1}$, ou bien ce
$\OO_K$-module est nul. Dans ce cas, le morphisme naturel
$$
\HH^{1}(X,L)\rightarrow \HH^{1}(X_1,L)
$$
n'est pas bijectif. Il existe donc un entier $m>1$, tel que
$L|_{X_m}\simeq \omega_{m}$ (corollaire \ref{condition pour L}),
par suite $L|_{X_1}\simeq \omega_{m}|_{X_1}\simeq
\mathcal{I}|_{X_1}^{i}$ pour $i$ un entier convenable (on peut
supposer $0\leq i<d_1$). De plus, pour $m$ un entier tel que
$1\leq m\leq \psi(2)-1 $, le morphisme canonique
$\Pic^{\circ}(X_m)\rightarrow \Pic^{\circ}(X_1)$ est bijectif
(\ref{Pic_S}), il en résulte que $L|_{X_m}\simeq
\mathcal{I}|_{X_m}^{i}$. Or $\psi(2)=\psi(1)+d_1=1+d_1$, et par
définition , $\omega_m=\omega_{X/S}(mD)|_{X_m}$, il existe donc un
unique entier $m$ tel que $1\leq m<\psi(2)-1$, et que $L\simeq
\omega_m$. Donc, en vertu de \ref{lemme cohomologie pour L}, on
trouve que le $\OO_K$-module $\HH^{1}(X_{\psi(2)-1},L)$ est de
longueur $1$. Supposons maintenant l'assertion vérifiée pour
$1\leq n'-1< n$. Sous l'hypothèse du lemme, le morphisme
$$
\HH^{1}(X,L)\rightarrow \HH^{1}(X_{\psi(n')-1},L)
$$
n'est pas bijectif. Par suite, il existe un entier $m\geq
\psi(n')$, tel que $L|_{X_m}\simeq \omega_{m}$, par conséquent,
$L|_{X_{\psi(n')}}\simeq \mathcal{I}_{X_{\psi(n')}}^{i}$ pour
$0\leq i<d_{n'}$. Or $\psi(n'+1)=\psi(n')+d_{n'}$, on en déduit
qu'il existe un unique entier $m$ tel que $\psi(n')\leq m\leq
\psi(n'+1)-1$, et que $L|_{X_{m}}\simeq \omega_m$, en particulier,
le $\OO_K$-module $\HH^{1}(X_{\psi(n'+1)-1}, L)$ est de longueur
$n'$. Ceci finit la récurrence, et le lemme en résulte aussitôt.
\end{proof}



\subsection{Etudes numériques.}\label{Etude numerique}

\subsubsection{} Gardons les notations de \ref{notations}. En
particulier, $f':X'\rightarrow S$ est le $S$-modèle propre
régulier minimal de $X_K'=\Pic^{\circ}_{X_{K}/K}$. D'après le
théorème 3.8 de \cite{LLR}, il existe un morphisme de
$\OO_K$-modules:
$$
\tau_X:\HH^{1}(X,\OO_{X})\rightarrow \HH^{1}(X',\OO_{X'})
$$
qui prolonge l'isomorphisme naturel sur la fibre générique. De
plus, son noyau est la torsion de $\HH^{1}(X,\OO_X)$, et les
$\OO_K$-modules $\ker(\tau_X)$ et $\mathrm{coker}(\tau_X)$ ont la
même longueur. Par dualité, on obtient le morphisme suivant
$$
\tau_{X}^{\vee}:\HH^{0}(X',\omega_{X'/S})\simeq
\left(\HH^{1}(X',\OO_{X'})\right)^{\vee}\rightarrow
\left(\HH^{1}(X,\OO_X)\right)^{\vee}\simeq
\HH^{0}(X,\omega_{X/S}).
$$
Or il existe un isomorphisme canonique
$f'_{\ast}\omega_{X'/S}\simeq \omega$ (\ref{Classique}), d'où le
morphisme canonique:
$$
\tau_{X}^{\vee}:\omega\simeq \HH^{0}(X',\omega_{X'/S})\rightarrow
\HH^{0}(X,\omega_{X/S})=f_{\ast}\omega_{X/S}.
$$
qui est injectif, mais avec un conoyau non trivial en général.

\subsubsection{} On déduit d'abord de l'injectivité du
morphisme $\tau_{X}^{\vee}$ que $\omega_{X/S}$ contient
$f^{\ast}\omega$, donc il existe un entier $\chi\geq 0$, tel que
$\omega_{X/S}=f^{\ast}\omega\otimes \mathcal{I}^{-\chi}$. Comme
$f_{\ast}\omega_{X/C}=\pi^{-[\chi/d]}\omega$, on obtient que
$\mathrm{coker}(\tau_{X}^{\vee})$ a pour longueur $[\chi/d]$.

\begin{prop} Gardons les notations ci-dessus. Alors
\begin{equation}\label{expression}
\chi=d\left((1-1/d))+k_0(1-1/p^{r})+\cdots+k_{r-1}(1-1/p)\right),
\end{equation}
où les $k_i$ sont les entiers définis dans \ref{ki}.
\end{prop}

\begin{proof} D'après \ref{ki}, on a $\phi(m_r)=1+k_0+\cdots
k_{r-1}$, et à partir de $n=m_r$, $\phi(n)=\phi(n-1)+1$ si et
seulement si $n-m_r$ est un multiple de $d=d'p^{r}$. En
particulier, si l'on pose $m_r=hd-a$ avec $0\leq a<d$, on a
$\phi(m_r)=\phi(hd)$. Posons $M=\mathrm{R}^{1}f_{\ast}\OO_X$, et
$T\subset M$ son sous-faisceau de torsion. Notons $L=M/T$, qui est
donc libre de rang $1$ sur $R$. On a
$\mathrm{R}^{1}f_{\ast}\left(\OO_{X}/\pi^{n}\OO_X\right)\simeq
\mathrm{R}^{1}f_{\ast}(\OO_{X_{nd}})=M/\pi^{n}M$. Donc pour $n\geq
h$, la longueur de $M/\pi^{n}M$ croît de $1$ avec $n$. C'est dire
que $T$ est annulé par $\pi^{h}$, et que
$\ell(M/\pi^{h}M)=\ell(T)+\ell(L/\pi^{n}L)$, d'où
\begin{equation}\label{phi(hd)}
\phi(hd)=\ell(T)+h.
\end{equation}
Or $\omega_{m_r}=\mathcal{I}^{-(\chi+m_r)}|_{X_{m_r}}$ est le
fibré inversible trivial, et comme $\mathcal{I}|_{X_{m_r}}$ est
d'ordre $d$, il existe un entier $\alpha$ tel que $\chi+m_r=\alpha
d $. D'où $\chi=(\alpha-h)d+a$, et donc on a
$\ell(T)=[\chi/d]=\alpha-h$. On déduit (en utilisant l'égalité
(\ref{phi(hd)})) que $ \ell(T)=[\chi/d]=\alpha-h=\phi(m_r)-h$.
D'où $\alpha=\phi(m_r)$, et
\begin{eqnarray*}
\chi& = & \phi(m_r)d-m_r \\ & =&
(1+k_0+\cdots+k_{r-1})d-(1+k_0d'+\cdots+k_{r-1}d'p^{r-1})
\\ &=&
d\left((1-1/d))+k_0(1-1/p^{r})+\cdots+k_{r-1}(1-1/p)\right).
\end{eqnarray*}
\end{proof}

\begin{corollary} Les conditions suivantes sont équivalentes:

(i) $X/S$ est cohomologiquement plat (en dimension $0$).

(ii) $\chi<d$.

(iii) $m_r=1$.

(iv) $\mathcal{I}|_{X_1}$ est d'ordre $d$.

\noindent De plus, si ces conditions sont réalisées, on a
$\chi=d-1$.
\end{corollary}

\begin{proof} Le $S$-schéma $X/S$ est cohomologiquement plat si et
seulement si $T=0$, \emph{i.e.}, si et seulement si
$\ell(T)=[\chi/d]=0$, ce qui équivant à dire que $\chi<d$, d'où
(i)$\Leftrightarrow$(ii). L'équivalence de (iii) et (iv) résulte
de la définition de $m_r$. Par ailleurs, d'après l'égalité
(\ref{expression}), $\chi<d$ si et seulement si
$k_0=\cdots=k_{r-1}=0$, par suite si et seulement si
$m_0=\cdots=m_r=1$. La dernière assertion résulte directement de
l'égalité (\ref{expression}) et (iii).
\end{proof}

\begin{remark}Une fois que les conditions équivalentes
ci-dessus soient réalisées, nous dirons que le torseur $X_K$ est
\emph{modérément ramifié}; sinon, on dira que $X_K$ est
\emph{sauvagement ramifié}. Donc, en vertu de \ref{Oort}, si
$(d,p)=1$, le torseur $X_K$ est automatiquement modéré. Mais la
réciproque n'est pas vraie en général (\cite{Raynaud} remarques
9.4.3 d)).
\end{remark}

\section{Filtrations et comparaisons}

Pour tout $n\geq 1$ un entier, on a un morphisme canonique de
groupes $\Pic^{0}(X)\rightarrow \Pic^{\circ}(X_n)$. On obtient
ainsi une filtration sur les points à valeurs dans $S$ du foncteur
de Picard $\Pic^{\circ}_{X/S}(S)=\Pic^{\circ}(X)$. D'autre part,
le groupe $J(S)$ des points à valeurs dans $S$ de $J$ est
naturellement filtré par les puissances de $\pi$ (\emph{i.e.} la
filtration donnée par le morphisme canonique de groupes
$J(S)\rightarrow J(S_n)$). Le but de cette section est d'étudier
le comportement de ces deux filtrations vis-à-vis du morphisme
naturel de foncteurs $q:\Pic^{\circ}_{X/S}\rightarrow J$. Les
résultats obtenus s'énoncent agréablement en terme de réalisations
de Greenberg (\ref{resultat final}).

\subsection{Structures pro-algébriques}\label{fil}

\subsubsection{}\label{grouppro} Dans cet article, on appelle \emph{un groupe pro-algébrique sur
$k$} un pro-objet dans la catégorie des $k$-schémas en groupes de
type fini. Donc, on n'adopte pas le point de vue de Serre
(\cite{SerrePro}) des groupes pro-algébriques, où l'on travaille à
isogénie radicielle près.

\subsubsection{} Soit $n\geq 1$ un entier. Considérons
$\Gr(\Pic^{\circ}_{X_{n}/S})$ la réalisation de Greenberg du
foncteur de Picard $\Pic^{\circ}_{X_n/S}$ (\ref{Greenberg Pic}).
Le morphisme naturel de foncteurs
$\Pic^{\circ}_{X_{n+1}/S}\rightarrow \Pic^{\circ}_{X_n/S}$ induit
un morphisme de $k$-schémas en groupes lisses
$\alpha_n:\mathrm{Gr}(\Pic^{\circ}_{X_{n+1}/S})\rightarrow
\mathrm{Gr}(\Pic^{\circ}_{X_{n}/S})$. On obtient ainsi un k-groupe
pro-algébrique (au sens de \ref{grouppro})
$\{(\Gr(\Pic_{X_{n}/S}^{\circ}), \alpha_n)\}_{n\geq 1}$. De plus,
d'après le lemme suivant, qui découle de \cite{Lipman} et de
\ref{Pic_S}, on sait que ce $k$-groupe pro-algébrique est
pro-lisse.

\begin{lemma}\label{lemme en greenberg} Gardons les notations précédentes.
Alors, le morphisme $\alpha_n$ est un morphisme lisse et surjectif
de $k$-groupes lisses connexes. De plus, ou bien $\alpha_n$ est un
isomorphisme, auquel cas, on a $\phi(n+1)=\phi(n)$, ou bien
$\ker(\alpha_n)$ est un $k$-vectoriel de dimension $1$, auquel cas
on a $\phi(n+1)=\phi(n)+1$.
\end{lemma}


\subsubsection{} D'autre part, à partir du $S$-schéma en groupes $J/S$,
on peut construire un $k$-groupe pro-algébriques pro-lisse
$\{(\Gr_{n}(J),\beta_n)\}_{n\geq 1}$ (\ref{Greenberg lisse}). Or
pour chaque entier $n\geq 1$, le morphisme
$q:\Pic_{X/S}^{\circ}\rightarrow J$ induit un morphisme de
foncteurs $\Pic_{X/S}^{\circ}\times_S S_n\rightarrow J\times_S
S_n$, d'où un morphisme de $k$-groupes algébriques
$$
\Gr(\Pic_{X_{nd}/S}^{\circ})\rightarrow \Gr_n(J).
$$
En particulier, on obtient un morphisme de k-groupes
pro-algébriques:
$$
\{(\Gr(\Pic_{X_{n}/S}^{\circ}), \alpha_n)\}_{n\geq 1}\rightarrow
\{(\Gr_{n}(J),\beta_n)\}_{n\geq 1}.
$$
En fait, on a un résultat plus précise (\ref{resultat final}).
Pour le démontrer, il nous faut d'abord un résultat de théorie des
intersections.

\subsection{Un résultat de théorie des intersections}\label{resultat intersection}
\begin{prop}\label{intersection} Soient $R$ un anneau de valuation discrète, $Z/R$ un
$R$-schéma lisse de type fini, à fibre spéciale $\underline{Z}$
irréductible, et notons $\xi$ le point générique de de la fibre
spéciale $\underline{Z}$ de $Z$.\footnote{D'une manière générale,
pour $Z$ un schéma sur un anneau de valuation discète, on désigne
par $\underline{Z}$ sa fibre spéciale.} Soit $\mathcal{M}$ un
faisceau cohérent de torsion sur $Z$. Supposons que $\mathcal{M}$
est de longueur $\ell$ en $\xi$.

(1) Soit $\alpha\in Z(S)$ une section de $Z/S$ tel que
$\alpha(S)\nsubseteq \mathrm{Supp}(\mathcal{M})$. Alors la
longueur du $R$-module $\alpha^{\ast}\mathcal{M}$ est $\geq \ell$.
En plus, il y a égalité si et seulement si le support de
$\mathcal{M}$ en $\alpha(s)$ est contenu dans $\underline{Z}$, et
si $\mathcal{M}$ est de Cohen-Macaulay en $\alpha(s)$.

(2) Supposons que le support schématique de $\mathcal{M}_K$ est un
diviseur effectif non trivial $H_{K}\subset Z_{K}$, notons
$H\subset Z$ son adhérence schématique dans $Z$ (qui est un
diviseur effectif relatif). Soit $\alpha:S\rightarrow Z$ une
section de $Z/S$ avec $\alpha(s)\in \underline{H}$, telle que
$\ell(\alpha^{\ast}\mathcal{M})=\ell+1$. Alors (a) $\mathcal{M}$
est de Cohen-Macaulay en $x$; (b) $H$ est régulier en $x$; (c) si
l'on note $\zeta$ le point générique de la composante irréductible
de $H$ passant par $\alpha(s)$, alors $\mathcal{M}$ est de
longueur $1$ en $\zeta$; (d) $H$ et $\alpha(S)$ se coupent
tranversalement en $\alpha(s)$.
\end{prop}

Pour la démontrer, on utilise le lemme suivant:

\begin{lemma}\label{lemme technique} Soient $Z=\mathrm{Spec}(A)$ un schéma noethérien local
régulier de dimension $2$, $\mathcal{M}$ un $\OO_Z$-module
cohérent de torsion tel que $\mathrm{Supp}(\mathcal{M})$ soit de
dimension $1$. Soient $H_1,\cdots, H_n$ les composantes
irréductibles réduites de $\mathrm{Supp}(M)$. Notons $\xi_i$ le
point générique de $H_i$, et $\ell_i$ la longueur de
$\mathcal{M}_{\xi_{i}}$ sur $\OO_{Z,\xi_i}$. Soit enfin $f\in A$
qui fait partie d'un système régulier de paramètres de $A$, tel
que $Z_1:=V(f)\subset Z$ ne soit pas contenu dans
$\mathrm{Supp}(M)$. Alors $\ell(\mathcal{M}/f\mathcal{M})\geq
\sum_{i=1}^{n}\ell_i$, avec égalité si et seulement si les
conditions suivantes soient remplies: (i) Pour chaque $i$, le
schéma $H_i$ est régulier, et coupe transversalement $Z_1$ dans
$Z$; (ii) Le $\OO_Z$-module $\mathcal{M}$ est de Cohen-Macaulay.
\end{lemma}

\begin{proof} On raisonne par récurrence sur $n$.
Commençons par le cas où $n=1$. Notons $\xi=\xi_1$ le point
générique de $\mathrm{Supp}(M)$, et $\ell=\ell_1$ la longueur de
$\mathcal{M}$ en $\xi$. Donc, $\mathcal{M}_{\xi}$ possède une
filtration par des sous-$\OO_{Z,\xi}$-modules:
$$
0=\mathcal{M}_{\xi,0}\subset
\mathcal{M}_{\xi,1}\subset\cdots\subset
\mathcal{M}_{\xi,\ell}=\mathcal{M}_{\xi},
$$
où les quotients successifs sont isomorphes à $k(\xi)$.
Définissons $\mathcal{M}_{i}$ comme l'image réciproque de
$\mathcal{M}_{\xi,i}$ par le morphisme canonique
$\mathcal{M}\rightarrow \mathcal{M}_{\xi}$, on obtient ainsi une
filtration de $\mathcal{M}$:
$$
0=\mathcal{M}_{0}\subset \mathcal{M}_{1}\subset\cdots\subset
\mathcal{M}_{\ell}=\mathcal{M},
$$
et donc une filtration de $\mathcal{M}/f\mathcal{M}$ (remarquons
que $C_{i}:=\mathcal{M}_{i}/\mathcal{M}_{i-1}$ est non nul, et il
est sans composantes immergées dès que $i\geq 2$):
$$
0=\mathcal{M}_{0}/f\mathcal{M}_0\subset
\mathcal{M}_{1}/f\mathcal{M}_1\subset\cdots\subset
\mathcal{M}_{\ell}/f\mathcal{M}_{\ell}=\mathcal{M}/f\mathcal{M},
$$
où les quotients successifs sont isomorphes à $C_i/fC_i\neq 0$.
Par suite $\ell(\mathcal{M}/f\mathcal{M})\geq \ell$. De plus,
$\ell(\mathcal{M}/f\mathcal{M})=\ell$, si et seulement si pour
chaque $i$ ($1\leq i\leq \ell$), le $\OO_{Z}$-module $C_i/fC_i$
est de longueur $1$ sur $\mathcal{O}_Z/f\OO_Z$. Donc, il suffit de
prouver que cette dernière condition équivaut à dire que $C_i$ est
de Cohen-Macaulay, à support schématique régulier coupant
transversalement $V(f)\hookrightarrow Z$. En effet, supposons
$\mathrm{Ann}(C_i)=(g_i)\subset A$, et soit $c\in C_i$ tel que
$c\notin fC_i$, alors le morphisme $\OO_Z/g_i\OO_Z\rightarrow C_i$
donné par $\overline{\lambda}\mapsto \lambda c$ (où $\lambda\in
\OO_Z$ est un relèvement de $\overline{\lambda}$ dans $\OO_Z$) est
un isomorphisme. Par suite, $\OO_{Z}/(g_i,f)$ est de longueur $1$
sur $\OO_{Z}/f\OO_Z$. Donc
$\mathrm{Supp}(C_i)=\mathrm{Supp}(M)_{\mathrm{red}}$ est régulier,
qui coupe transversalement $V(f)\hookrightarrow Z$. Ceci finit la
preuve dans le cas où $n=1$. Supposons maintenant l'assertion
vérifiée pour $n-1\geq 1$. Soit $\mathcal{M}'\subset \mathcal{M}$
formé des $m\in \mathrm{M}$, tel que $\xi_i\notin
\mathrm{Supp}(m)$ pour chaque $i\geq 2$, et définissons
$\mathcal{M}''$ par la suite exacte suivante:
$$
0\rightarrow \mathcal{M}'\rightarrow \mathcal{M}\rightarrow
\mathcal{M}''\rightarrow 0.
$$
Alors $\mathcal{M}''$ est sans composantes immergées, et est à
support $\cup_{i=2}^{n}H_i$. On en déduit la suite exacte suivante
(car $V(f)\nsubseteq \mathrm{Supp}(\mathcal{M})$):
$$
0\rightarrow \mathcal{M}'/f\mathcal{M}'\rightarrow
\mathcal{M}/\mathcal{M}\rightarrow
\mathcal{M}''/f\mathcal{M}''\rightarrow 0.
$$
Donc, on a
$\ell(\mathcal{M}/f\mathcal{M})=\ell(\mathcal{M}'/f\mathcal{M}')+
\ell(\mathcal{M}''/f\mathcal{M}'')$. Le lemme en résulte aussitôt.
\end{proof}

\begin{proof}[Démonstration de \ref{intersection}] Posons
$x=\alpha(s)$. Il existe des éléments $f_1,\cdots,f_d$ de l'idéal
maximal de $\OO_{Z,x}$, qui engendrent l'idéal maximal de
$\OO_{\underline{Z},x}$, et tel que $\alpha(S)=V(f_1,\cdots,
f_d)\hookrightarrow Z$. Quitte à remplacer $Z$ par son localisé en
$x:=\alpha(s)$, on peut supposer $Z$ local de dimension $d+1$. Le
cas où $d=0$ est trivial, et le cas où $d= 1$ résulte directement
du lemme \ref{lemme technique}. On suppose désormais $d\geq 2$, et
on raisonne par récurrence sur $d$. Notons $Z_1\hookrightarrow Z$
le sous-schéma fermé défini par l'équation $f_1=0$,
$\mathcal{M}_1$ l'image réciproque de $\mathcal{M}$ sur $Z_1$.
Alors $Z_1$ est un schéma régulier local de dimension
$\mathrm{dim}(Z)-1$. Notons $\xi_1$ le point générique de
$\underline{Z_1}$. Le morphisme $\alpha:S\rightarrow Z$ se
factorise à travers $Z_1\hookrightarrow Z$, et on désigne par
$\alpha_1:S\rightarrow Z_1$ le morphisme ainsi obtenu.  Pour
démontrer la première assertion de (1), il suffit de vérifier que
$\mathcal{M}_1$ est de longueur $\geq \ell$ en $\xi_1$. Pour cela,
quitte à remplacer $Z$ par son localisé en $\xi_1$, on est amené
au cas où $d=1$, d'où l'assertion (\ref{lemme technique}).
Supposons maintenant
$\ell(\alpha^{\ast}\mathcal{M})=\ell(\alpha_{1}^{\ast}\mathcal{M}_1)=\ell$.
Donc, en vertu de l'hypothèse de récurrence, ceci équivaut à dire
que $\mathcal{M}_1$ est de longueur $\ell$ en $\xi_1$, et que
$\mathcal{M}_1$ est de Cohen-Macaulay, à support dans
$\underline{Z_1}$. Autrement dit, $\mathcal{M}$ est de
Cohen-Macaulay, et à support dans $\underline{Z}$. D'où (1). Pour
démontrer (2), comme c'est une question locale pour la topologie
étale sur $S$, on peut supposer $S$ strictement local, de sorte
que le corps résiduel $k(s)$ de $S$ est un corps infini. Il en
résulte que le corps résiduel $k(x)$ de $Z$ en $x$ est aussi
infini. Comme $k(x)$ est un corps infini, quitte à remplacer $f_1$
par $f_1+\lambda f_2$ avec $\lambda\in \OO_{Z,x}^{\ast}$ un
élément convenable, on peut supposer $\underline{Z_1}\nsubseteq
\underline{H}$, de sort que $\underline{H_1}=\underline{H}\cap
\underline{Z_1}\hookrightarrow \underline{Z_1}$ est de codimension
$1$ dans $\underline{Z_1}$ (où $H_1:=H\cap Z_1$). Puisque
$\ell(\alpha^{\ast}\mathcal{M})=\ell(\alpha_{1}^{\ast}\mathcal{M}_1)=\ell+1$,
par l'hypothèse de récurrence, on a (i) $H_{1,\mathrm{red}}$ est
irréductible et régulier, de plus, $\alpha_1(S)$ et $H_1$ se
coupent transversalement dans $Z_1$; (ii) $\mathcal{M}_1$ est de
Cohen-Macaulay dans $Z_1$, et si l'on note $\zeta_1\in H_1$ le
point générique de $H_1$, alors $\mathcal{M}_1$ est de longueur
$1$ en $\zeta_1$. Notons ensuite $Z'$ le localisé de $Z$ en
$\zeta_1$, et $\mathcal{M}'$ l'image réciproque de $\mathcal{M}$
par le morphisme canonique $Z'\rightarrow Z$, par suite
$\mathcal{M}'/f_1\mathcal{M}'$ est de longueur $1$ sur
$\OO_{Z'}/f_1\OO_{Z'}$. Donc, le lemme \ref{lemme technique}
implique que $H$ est régulier en $\zeta_1$, coupant
tranversalement $Z_1$ en $\zeta_1$. De plus, $\mathcal{M}$ est de
Cohen-Macaulay, à support contenu dans $H$ en $\zeta_1$. Par
conséquent, $H_1$ est réduit. Compte tenu de (i) ci-dessus, on en
déduit que $H$ est irréductible et régulier, coupant
tranversalement $Z_1$ dans $Z$. Il reste à vérifier que
$\mathcal{M}$ est de Cohen-Macaulay. Par (ii), il suffit de
prouver que $\mathcal{M}$ est sans composantes immergées. Notons
$\mathcal{M}'$ le plus grand qotient sans composantes immergées de
$\mathcal{M}$, et définissons $\mathcal{N}$ le sous-$\OO_Z$-module
par la suite exacte suivante:
$$
0\rightarrow \mathcal{N}\rightarrow \mathcal{M}\rightarrow
\mathcal{M}'\rightarrow 0.
$$
Alors le $\OO_{Z}$-module $\mathcal{M}'$ satisfait les hypothèses
de la proposition, donc, d'après ce qui précède,
$\mathcal{M}'/f_1\mathcal{M}'$ est de Cohen-Macaulay, avec
$\ell(\alpha^{\ast}\mathcal{M}')=\ell+1$. On a ainsi une suite
exacte
$$
0\rightarrow \alpha^{\ast}\mathcal{N}\rightarrow
\alpha^{\ast}\mathcal{M}\rightarrow
\alpha^{\ast}\mathcal{M}'\rightarrow 0.
$$
D'où $\alpha^{\ast}\mathcal{N}=0$ pour la raison de longueur. Par
suite $\mathcal{N}=0$. Ceci achève la preuve.
\end{proof}

\subsection{Comparaison de structures pro-algébriques}

\subsubsection{} Le but de cette sous-section est de démontrer que
le morphisme canonique de foncteurs
$q:\Pic^{\circ}_{X/S}\rightarrow J$ induit, pour chaque $n\geq 1$,
un morphisme de $k$-groupes algébriques lisses
$q_n:\mathrm{Gr}(\Pic^{\circ}_{X_{\psi(n)}/S})\rightarrow
\mathrm{Gr}_n(J)$, rendant commutatif le diagramme évident.

Soit $n\geq 1 $ un entier. Il existe alors un morphisme de
foncteurs $\Pic^{\circ}_{X_{nd}/S_n}\rightarrow J\times_S S_n$ (où
$S_n=\mathrm{Spec}(\OO_{K}/\pi^{n})$), d'où un morphisme de
$k$-schémas en groupes lisses
\begin{equation}\label{mor}
\mathrm{Gr}(\Pic^{\circ}_{X_{nd}/S})\rightarrow \mathrm{Gr}_{n}(J)
\end{equation}
Puisque $\psi(n)\!\!\leq\!\! nd$ (\ref{def de psi}), il y a un
morphisme canonique de $k$-groupes algébriques lisses
$$
\Gr(\Pic^{\circ}_{X_{nd}/S})\rightarrow
\Gr(\Pic^{\circ}_{X_{\psi(n)}/S}).
$$
Donc, pour démontrer
l'assertion ci-dessus, il suffit de vérifier que le morphisme
(\ref{mor}) ci-dessus se factorise à travers le morphisme
canonique $\mathrm{Gr}(\Pic^{\circ}_{X_{nd}/S})\rightarrow
\mathrm{Gr}(\Pic^{\circ}_{X_{\psi(n)}/S})$:
$$
\xymatrix{\mathrm{Gr}(\Pic^{\circ}_{X_{nd}/S})\ar[r]\ar[d]&
\mathrm{Gr}_{n}(J)\\
\mathrm{Gr}(\Pic^{\circ}_{X_{\psi(n)}/S})\ar@{.>}[ru]& }
$$
D'autre part, comme le morphisme de $k$-groupes algébriques
$\mathrm{Gr}(\Pic^{\circ}_{X_{nd}/S})\rightarrow
\mathrm{Gr}(\Pic^{\circ}_{X_{\psi(n)}/S})$ est à noyau lisse
(\ref{lemme en greenberg}), il suffit de le vérifier au niveau des
$k$-points rationnels (rappelons que $k=\overline{k}$ est
algébriquement clos). Donc, on est amené à montrer que le
morphisme canonique (sur les points à valeurs dans $S$)
$\Pic^{\circ}_{X/S}(S)=\Pic^{\circ}(X)\rightarrow J(S_n)$ se
factorise à travers le morphisme canonique de groupes abstraits
$\Pic^{\circ}(X)\rightarrow \Pic^{\circ}(X_{\psi(n)})$ (on notera
encore par $q_n$ le morphisme de groupes abstraits ainsi obtenu):
$$
\xymatrix{\Pic^{\circ}(X)\ar[r]\ar[d]_{q}&
\Pic^{\circ}(X_{\psi(n)})\ar@{.>}[d]^{q_n}\\ J(S)\ar[r]& J(S_n)}.
$$
On va établir cette factorisation à l'aide du foncteur de Picard
rigidifié (rappelé dans $\S$\ref{recall on Picard}).

\subsubsection{}\label{dilatation}Dans la suite, on utilisera
la notion de {\textquotedblleft dilatation\textquotedblright} d'un
schéma en groupes. Rappelons brièvement cette construction (3.2 de
\cite{BLR}). Soit $R$ un anneau de valuation discrète de corps
résiduel $\kappa$, avec $\pi\in R$ une uniformisante. Soient $H/R$
un $R$-schéma en groupes lisse de type fini, $W\hookrightarrow
\underline{H}$ un sous-schéma en groupes \emph{lisse} sur
$\kappa$. Notons $\mathcal{J}$ l'idéal de définition de
$W\hookrightarrow H$. Posons $\mathrm{Bl}_{W}(H)$ l'éclatement de
$H$ le long du centre $W\hookrightarrow H$. Alors, par définition
(\cite{BLR} 3.2), \emph{la dilatation de $H$ le long du centre
$W\hookrightarrow H$} est le plus grand ouvert $H'\subset
\mathrm{Bl}_W(H)$ tel que l'idéal $\mathcal{J}\OO_{H'}\subset
\OO_{H'}$ soit engendré par $\pi$. D'après \cite{BLR} 3.2/3, $H'$
est un $R$-schéma en groupes \emph{lisse}, vérifiant la propriété
universelle suivante (\cite{BLR} 3.2/1): soient $Z/R$ un
$R$-schéma \emph{plat}, et $v:Z\rightarrow H$ un morphisme de
$R$-schémas tel que sa fibre spéciale
$\underline{v}:\underline{Z}\rightarrow \underline{H}$ se
factorise à travers $W\hookrightarrow \underline{H}$, alors il y
existe un unique morphisme $v':Z\rightarrow H'$ rendant commutatif
le diagramme évident.

\subsubsection{}\label{notation} On va utiliser les notations
suivantes. Pour $n\leq m$ deux entiers $\in \mathbf{Z}_{\geq
0}\cup \{\infty\}$, notons $\mathrm{P}^{[n,m]}$ le noyau du
morphisme canonique de foncteurs $\Pic^{\circ}_{X_m/S}\rightarrow
\Pic^{\circ}_{X_n/S}$. Ici, par convention, on note
$X_{\infty}=X$, et $\Pic^{\circ}_{X_0/S}=\{0\}$ (l'objet final de
la catégorie des faisceaux abéliens fppf sur $S$). Pour faciliter
les notations, posons
$$
\PP^{[n]}:=\PP^{[n,\infty]}=\ker(\Pic^{\circ}_{X/S}\rightarrow
\Pic^{\circ}_{X_n/S}), ~~~~~~~~
\PP_{[n]}:=\PP^{[0,n]}=\Pic^{\circ}_{X_n/S}.
$$
Pour chaque entier $n\geq 1$, on définit par récurrence un
$S$-schéma en groupes lisse $J^{[n]}$ comme la dilatation
(\ref{dilatation}) de $J^{[n-1]}$ le long de l'élément neutre de
$\underline{J^{[n-1]}}$ (ici, $J^{[0]}:=J$). En vertu de la
propriété universelle des dilatations (3.2/1 de \cite{BLR}), pour
chaque $n\in \mathbf{Z}_{\geq 0}$, on a la suite exacte suivante:
$$
0\rightarrow J^{[n]}(S)\rightarrow J(S)\rightarrow
J(S_n)\rightarrow 0.
$$
D'où un diagramme commutatif à lignes exactes:
$$
\xymatrix{0\ar[r]& J^{[n]}(S)\ar[r]\ar[d]& J(S)\ar[r]\ar@{=}[d]&
J(S_{n})\ar[r]\ar[d]& 0 \\ 0\ar[r]& J^{[n-1]}(S)\ar[r]&
J(S)\ar[r]& J(S_{n-1})\ar[r]& 0}
$$
Par conséquent, le morphisme canonique de groupes abstraits
\begin{equation}\label{iso}
J^{[n-1]}(S_1)\simeq \mathrm{coker}(J^{[n]}(S)\rightarrow
J^{[n-1]}(S))\rightarrow \ker(J(S_{n})\rightarrow J(S_{n-1}))
\end{equation}
est un isomorphisme.

\subsubsection{}\label{Not} Fixons une fois pour toutes
un rigidificateur $Y\hookrightarrow X$ pour le
foncteur de Picard relatif $\Pic_{X/S}$ (\ref{rigidificateur}), et
pour simplifier les notations, désignons par
$G=(\Pic_{X/S},Y)^{\circ}$ la composante neutre du foncteur de
Picard relatif de $X/S$, rigidifié le long $Y/S$ (\ref{def de Q}),
et par $J$ la composante neutre du $S$-modèle de Néron de
$\Pic^{\circ}_{X_K/K}$. D'après la proposition 3.2 de \cite{LLR},
$G$ est représentable par un $S$-schéma en groupes lisse séparé.
Considérons le morphisme canonique de $S$-schémas en groupes $ r:
G=(\Pic_{X/S},Y)^{\circ}\rightarrow \Pic^{\circ}_{X/S}$ (rappelé
dans $\S$\ref{presentation de P}), qui est surjectif pour la
topologie étale. Comme $S$ est strictement local, le morphisme $r$
induit une surjection sur les $S$-points (encore noté par $r$):
$$
r:G(S)\rightarrow \Pic^{\circ}(X).
$$
Notons $N=\overline{\ker(r_{K})}$ l'adhérence schématique de
$\ker(r_{K})\subset G_{K}$ dans $G$. C'est un $S$-schéma en
groupes plat de type fini, qui est aussi le noyau du morphisme
canonique $\theta: G\rightarrow J$ ($=$ composé de $r:G\rightarrow
\Pic^{\circ}_{X/S}$ et de l'épimorphisme fppf
$\Pic^{\circ}_{X/S}\rightarrow J$). On fixe aussi $\mathfrak{L}$
un faisceau de Poincaré (rigidifié) sur $X\times G$. Pour chaque
$n\in \mathbf{Z}_{\geq 1}$, notons
$$
r_n:G(S)\rightarrow \Pic^{\circ}(X)\rightarrow \Pic^{\circ}(X_n)
$$
le morphisme canonique de groupes abstraits, qui est donné par
$\varepsilon \in G(S) \mapsto \mathcal{L}_{\varepsilon}|_{X_n}\in
\Pic^{\circ}(X_n)$, où
$\mathcal{L}_{\varepsilon}:=(\mathrm{id}_{X}\times
\varepsilon)^{\ast}\mathfrak{L}$.


\subsubsection{} Notons
$$
p_G:X\times_S G\rightarrow G
$$
la projection de $X\times_S G$ sur le deuxième facteur.
Considérons $\mathrm{R}p_{G,\ast}\mathfrak{L}$. C'est un complexe
parfait d'amplitude parfaite contenue dans $[0,1]$. Donc,
localement pour la topologie Zariski de $G$, $\R
p_{G,\ast}\mathcal{P}$ peut se représenter par un complexe
$$
\xymatrix{\cdots\ar[r]& \mathcal{F}^{0}\ar[r]^{u}&
\mathcal{F}^{1}\ar[r]& \cdots}.
$$
avec $\mathcal{F}^{i}$ ($i=0,1$) des $\OO_G$-modules localement
libres de même rang. Or pour $L$ un faisceau inversible de degré
$0$ sur $X_{K}$, $\HH^{1}(X_{K},L)\neq 0$ si et seulement si
$L\simeq \OO_{X_K}$. Par suite, $\mathrm{det}(u)\neq 0$, et le
morphisme $u:\mathcal{F}^{0}\rightarrow \mathcal{F}^{1}$ est
injectif. Notons
$\mathcal{M}:=\mathrm{R}^{1}p_{G,\ast}\mathfrak{L}=\mathrm{coker}(u)$.
C'est donc un $\OO_G$-module de torsion, qui possède une
résolution de longueur $1$ par des $\OO_G$-modules localement
libres. En particulier, le $\OO_G$-module $\mathcal{M}$ est de
Cohen-Macaulay, à support $\mathrm{Supp}(\mathcal{M})\subset N\cup
\underline{G}$ (comme ensemble).

\subsubsection{}\label{longueur 0} Notons $\xi\in \underline{G}$ le point générique de
$\underline{G}$, et $\ell$ la longueur de $\OO_{G,\xi}$-module
$\mathcal{M}_{\xi}$. Alors $\ell=0$. Pour le montrer, on raisonne
par l'absurde. Supposons $\ell\geq 1$. Soit $\varepsilon\in G(S)$
une section de $G$, et notons
$\mathcal{L}_{\varepsilon}=(\mathrm{id}_{X}\times
\varepsilon)^{\ast}\mathfrak{L}$. En vertu de \ref{intersection},
le $\OO_K$-module $\HH^{1}(X,\mathcal{L}_{\varepsilon})\simeq
\varepsilon^{\ast}\mathcal{M}$ est de longueur $\geq \ell\geq 1$.
D'après \ref{cle}, ceci équivaut à dire que
$\mathcal{L}_{\varepsilon}|_{X_1}\simeq \mathcal{I}^{i}|_{X_1}$
avec $i$ un entier convenable. Il en résulte que le morphisme
surjectif de groupes
$$
r_1:G(S)\rightarrow \Pic^{\circ}(X_1), ~~~~\varepsilon\mapsto
\mathcal{L}_{\varepsilon}|_{X_1}
$$
est à image finie. Mais $\OO_K$ est à corps résiduel
algébriquement clos, en particulier, $\Pic^{\circ}(X_1)\simeq
\Pic^{\circ}_{X_1/k}(k)$ est un groupe \emph{infini}, d'où une
contradiction, et l'assertion en résulte aussitôt. On en déduit
que le support (comme ensemble) du $\OO_G$-module $\mathcal{M}$
est le fermé $N$ de $G$.

\subsubsection{} Commençons par comparer le premier cran des filtrations.
Puisque $X_1/S$ est définissable sur son point fermé, son foncteur
de Picard $\Pic^{\circ}_{X_1/S}$ l'est aussi. Donc, par
adjonction, le morphisme de foncteurs $r_1:G\rightarrow \PP_{[1]}$
correspond à un morphisme de groupes algébriques sur le point
fermé $s$ de $S$:
$$
\underline{r_1}:\underline{G}\rightarrow
\underline{\PP_{[1]}}=\Pic^{\circ}_{X_1/k},
$$
rendant le diagramme suivant commutatif:
$$
\xymatrix{G\ar[r]\ar[d] & \Pic_{X/S}^{\circ}\ar[r] & \PP_{[1]}=i_{\ast}\underline{\PP_{[1]}}\\
i_{\ast}\underline{G}\ar@{.>}[rru]_<<<<<<<<<<<<<<{i_{\ast}\underline{r_1}}
& & }.
$$
Soit $x\in \underline{G}(k)$ un point fermé de $G$, et
$\varepsilon\in G(S)$ un relèvement de $x$. Notons
$\mathcal{L}_{\varepsilon}=(\mathrm{id}_{X}\times
\varepsilon)^{\ast}\mathfrak{L}$. C'est un faisceau inversible
rigidifié sur $X$. Alors $\varepsilon(s)\in N$ si et seulement si
le $\OO_K$-module $\HH^{1}(X,\mathcal{L}_{\varepsilon})$ est de
longeur $\geq 1$ (\ref{intersection}). De plus, cette dernière
condition, en vertu du lemme \ref{cle}, revient à dire que
$\mathcal{L}_{\varepsilon}|_{X_1}\simeq \mathcal{I}^{i}|_{X_1}$
avec $i$ un entier convenable. Notons $Z$ la fermeture schématique
des $x\in G(k)$ qui admettent un relèvement $\varepsilon$ tel que
$\mathcal{L}_{\varepsilon}|_{X_1}\simeq \OO_{X_1}$. Par continuité
et connexité, $Z$ est donc une réunion de composantes
irréductibles de $\underline{N}_{\mathrm{red}}$. Si l'on note
$G^{[1]}\rightarrow G$ la dilatation de $G$, de centre $Z\subset
\underline{G}$. Par définition de $Z$, on a une suite exacte de
$k$-schémas en groupes lisses:
$$
0\rightarrow Z\rightarrow \underline{G}\rightarrow
\Pic^{\circ}_{X_1/k}\rightarrow 0.
$$
D'où la suite exacte suivante, en vertu de la propriété
universelle des dilatations (3.2/1, \cite{BLR}):
\begin{equation}\label{exact de P}
0\rightarrow G^{[1]}(S)\rightarrow G(S)\rightarrow
\Pic^{\circ}(X_1)\rightarrow 0.
\end{equation}

\subsubsection{} Regardons ensuite le morphisme $\theta:G\rightarrow J$.
Notons $G^{[1]'}$ la dilatation de $G$ le long du centre
$\underline{N}_{\mathrm{red}}=\underline{\mathrm{ker}(\theta)}_{\mathrm{red}}\hookrightarrow
\underline{G}$. La propriété universelle des dilatations (3.2/1 de
\cite{BLR}) entraîne que la suite suivante est exacte
\begin{equation}\label{exact de J}
0\rightarrow G^{[1]'}(S)\rightarrow G(S)\rightarrow
J(S_1)\rightarrow 0.
\end{equation}
Comme $Z\subset \underline{N}_{\mathrm{red}}$ est un sous-groupe
ouvert, $G^{[1]}$ est un sous-groupe ouvert de $G^{[1]'}$. Compte
tenu des suites exactes (\ref{exact de P}) (\ref{exact de J}), on
obtient un morphisme de groupes $q_1:\Pic^{\circ}(X_1)\rightarrow
J(S_1)$, rendant le carré suivant commutatif:
$$
\xymatrix{\Pic^{\circ}(X)\ar[r]\ar[d]_{q}& \Pic^{\circ}(X_1)\ar[d]^{q_1}\\
J(S)\ar[r]& J(S_1)}
$$
Le morphisme de groupes abstraits $q_1$ est surjectif, de noyau
engendré par $\mathcal{I}|_{X_1}\in \Pic(X_1)$ (\ref{cle}).

\subsubsection{}\label{fidele} Passons au deuxième cran des filtrations.
Notons $\mathcal{M}^{[1]}$ (resp. $\mathcal{M}^{[1]'}$) l'image
réciproque de $\mathcal{M}$ sur $G^{[1]}$ (resp. sur $G^{[1]'}$)
via le morphisme $G^{[1]}\rightarrow G$ (resp. via le morphisme
$G^{[1]'}\rightarrow G$). Notons $N^{[1]}$ (resp. $N^{[1]'}$)
l'adhérence schématique de $N_{K}\hookrightarrow
G_{K}^{[1]}=G_{K}$ dans $G^{[1]}$ (resp. dans $G^{[1]'}$). Alors
$\mathcal{M}^{[1]}$ (resp. $\mathcal{M}^{[1]'}$) est un faisceau
cohérent de torsion à support dans $N^{[1]}\cup
\underline{G^{[1]}}$ (resp. dans $N^{[1]'}\cup
\underline{G^{[1]'}}$). De plus, d'après la propriété universelle
des dilatations (3.2/1 de \cite{BLR}), le morphisme composé
$G^{[1]}\rightarrow G\rightarrow J$ (resp. $G^{[1]'}\rightarrow
G\rightarrow J$) se factorise à travers $J^{[1]}\rightarrow J$. On
désigne par $\theta^{[1]}:G^{[1]}\rightarrow J^{[1]}$ (resp.
$\theta^{[1]'}:G^{[1]'}\rightarrow J^{[1]}$) le morphisme ainsi
obtenu.

\begin{lemma}\label{pf rec} Gardons les notations ci-dessus.

(i) Soit $\xi_1'$ un point générique de $\underline{G^{[1]'}}$,
alors le $\OO_{G^{[1]'},\xi_1'}$ module
$\mathcal{M}^{[1]'}_{\xi_1'}$ est de longeur $1$.

(ii) Le schéma $N$ est normal.

(iii) Le morphisme $\theta^{[1]}$ induit une surjection
$G^{[1]}(S)\rightarrow J^{[1]}(S)$. En particulier,
$\theta^{[1]}:G^{[1]}\rightarrow J^{[1]}$ est un morphisme
fidèlement plat, et $\ker(\theta^{[1]})=N^{[1]}$.
\end{lemma}

\begin{proof} Remarquons d'abord que l'on a un diagramme
commutatif à lignes exactes:
$$
\xymatrix{0\ar[r]& G^{[1]}(S)\ar[r]\ar[d]& G(S)\ar[r]\ar[d] &
\Pic^{\circ}(X_1)\ar[r]\ar@{=}[d] & 0 \\ 0\ar[r] &
\PP^{[1,\psi(2)]}(S)\ar[r]& \PP_{[\psi(2)]}(S)\ar[r]&
\PP_{[1]}(S)\ar[r]& 0 }.
$$
On en déduit que le morphisme $G^{[1]}(S)\rightarrow
\PP^{[1,\psi(2)]}(S)$ est surjectif, car le morphisme
$G(S)\rightarrow \PP_{[\psi(2)]}(S)$ l'est. De plus, d'après
\ref{Oort}, le groupe $\PP^{[1,\psi(2)]}(S)$ est un $\OO_K$-module
de longueur $1$. Donc il est un groupe infini. Par suite, le
morphisme composé $ G^{[1]'}(S)\rightarrow G(S)\rightarrow
\PP_{[\psi(2)]}(S)=\Pic^{\circ}({X_{\psi(2)}}) $ est à image
\emph{infinie} (rappelons que $G^{[1]}$ est un sous-groupe ouvert
de $G^{[1]'}$). Pour montrer (i), on raisonne par l'absurde.
Notons $\ell_1$ la longueur du $\OO_{G^{[1]'},\xi_1'}$-module
$\mathcal{M}^{[1]'}_{\xi_1'}$. Par définition de $G^{[1]'}$, on a
$\ell_1\geq 1$. Supposons $\ell_1\geq 2$. Soit $\varepsilon\in
G^{[1]'}(S)$ une section de $G^{[1]'}$, avec
$\mathcal{L}_{\varepsilon}$ le faisceau inversible rigidifié sur
$X$ associé, telle que $\varepsilon(s)\in
\overline{\{\xi_1'\}}\subset \underline{G^{[1]'}}$. En vertu de la
proposition \ref{intersection}, le $\OO_K$-module
$\varepsilon^{\ast}\mathcal{M}^{[1]'}\simeq
\HH^{1}(X,\mathcal{L}_{\varepsilon})$ est de longueur $\geq
\ell_1\geq 2$. Donc, d'après \ref{cle}, on a
$\mathcal{L}_{\varepsilon}|_{X_{\psi(2)}}\simeq
\mathcal{I}^{i}|_{X_{\psi(2)}}$ avec $i$ un entier convenable.
Comme $G^{[1]'}/S$ est un $S$-schéma en groupes lisse de type
fini, on en déduit que le morphisme $G^{[1]'}(S)\rightarrow
\PP_{[\psi(2)]}(S)=\Pic^{\circ}(X_{\psi(2)})$ est à image finie.
D'où une contradiction, et ceci achève la preuve de (i).
L'assertion (ii) est un corollaire de (i), en vertu de
\ref{intersection} (2). Pour démontrer (iii), rappelons que le
morphisme composé $G(S)\rightarrow \Pic^{\circ}(X)\rightarrow
J(S)$ est surjectif (\ref{Brauer nul}). Puisque $G^{[1]'}$ est la
dilatation de $G$ le long de $\underline{N}_{\mathrm{red}}$, la
surjectivité de la dernière flèche implique la surjectivité de
$G^{[1]'}(S)\rightarrow J^{[1]}(S)$. Puisque $G^{[1]}\subset
G^{[1]'}$ est un sous-groupe ouvert, à fibre spéciale non vide, le
groupe abstrait $G^{[1]'}(S)/G^{[1]}(S)$ est un groupe fini. Par
suite, en vertu de 9.2/6 de \cite{BLR}, $\theta^{[1]}$ induit
aussi une surjection $G^{[1]}(S)\rightarrow J^{[1]}(S)$. En
particulier, $\theta^{[1]}$ est un morphisme fidèlement plat, et
on a $\ker(\theta^{[1]})=N^{[1]}$.
\end{proof}

\subsubsection{} Considérons le morphisme composé
$$
\xymatrix{\beta_{2}:G^{[1]}\ar[rr]^{r^{[1]}}(S) &  &
\PP^{[1]}(S)\ar[rr] &  & \PP^{[1,\psi(2)]}(S).}
$$
qui est surjectif. Puisque $\psi(2)\!\!\leq \!\! 2d$, le morphisme
de foncteurs $G^{[1]}\rightarrow \PP_{[\psi(2)]}$ définit un
morphisme de $k$-schémas en groupes
$\mathrm{Gr}_{2}(G^{[1]})\rightarrow
\mathrm{Gr}(\PP_{[\psi(2)]})$. Montrons d'abord que ce dernier
morphisme se factorise à travers le morphisme canonique de
$k$-groupes algébriques $\mathrm{Gr}_2(G^{[1]})\rightarrow
\underline{G^{[1]}}=\Gr_1(G^{[1]})$. Puisque le morphisme
$\Gr_2(G^{[1]})\rightarrow \underline{G^{[1]}}$ est à noyau lisse
sur $k$, il suffit de le vérifier au niveau des $k$-points
rationnels. Soit $\varepsilon \in G^{[1]}(S)$, avec
$\mathcal{L}_{\varepsilon}$ le faisceau inversible (rigidifié) sur
$X$ correspondant, tel que $\varepsilon(s)$ est l'élément neutre
de $\underline{G^{[1]}}$. En vertu de \ref{pf rec} (1) et de
\ref{intersection}, le $\OO_{K}$-module
$\varepsilon^{\ast}\mathcal{M}^{[1]}\simeq
\HH^{1}(X,\mathcal{L}_{\varepsilon})$ est de longueur $>1$. Donc,
par \ref{cle}, on a
$\mathcal{L}_{\varepsilon}|_{X_{\psi(2)}}\simeq
\mathcal{I}^{i}|_{X_{\psi(2)}}$ pour $i$ un entier convenable.
D'autre part, comme le noyau du morphisme (de $k$-groupes
algébriques) $\Gr_2(G^{[1]})\rightarrow \underline{G^{[1]}}$ est
connexe (\ref{Greenberg lisse}), on a forcément
$\mathcal{L}_{\varepsilon}|_{X_{\psi(2)}}\simeq
\OO_{X_{\psi(2)}}$. D'où l'assertion, et on désigne par
$\underline{\beta_2}:\underline{G}^{[1]}\rightarrow
\Gr(\PP_{[\psi(2)]})$ le morphisme de $k$-groupes algébriques
ainsi obtenu. Notons ensuite
$Z:=\ker(\underline{\beta_2})_{\mathrm{red}}\hookrightarrow
\underline{G^{[1]}}$. Alors  le même raisonnement implique que $Z$
est une réunion de composantes connexes de
$\underline{N^{[1]}}_{\mathrm{red}}$.

\subsubsection{}\label{preuve normal} Définissons $G^{[2]}$ (resp.
$G^{[2]'}$) comme la dilatation de $G^{[1]}$ le long du
sous-groupe fermé lisse $Z$ de $\underline{G^{[1]}}$ (resp. le
long de $\underline{N^{[1]}}_{\mathrm{red}}\hookrightarrow
\underline{G^{[1]}}$), et notons $\alpha^{[1]}$ le morphisme
composé $G^{[2]}\rightarrow G^{[1]}\rightarrow G$. D'après 3.2/3
de \cite{BLR}, $G^{[2]}$ est un $S$-schéma en groupes lisse, et on
a une suite exacte:
\begin{equation}\label{s}
0\rightarrow G^{[2]}(S)\rightarrow G^{[1]}(S)\rightarrow
\PP^{[1,\psi(2)]}(S)\rightarrow 0.
\end{equation}
D'autre part, on a une suite exacte de groupes abstraits
$$
0\rightarrow G^{[2]'}(S)\rightarrow G^{[1]}(S)\rightarrow
J^{[1]}(S_1)\rightarrow 0.
$$
Comme $Z\subset \underline{N^{[1]}}$ est un sous-schéma en groupes
ouvert, $G^{[2]}\subset G^{[2]'}$ est un sous-groupe ouvert. On
obtient donc un morphisme $\PP^{[1,\psi(2)]}(S)\rightarrow
J^{[1]}(S_1)$ rendant le diagramme suivant commutatif:
$$
\xymatrix{G^{[1]}(S)\ar[r]\ar[rd]& \PP^{[1,\psi(2)]}(S)\ar[d] \\
& J^{[1]}(S_1)}
$$
Comme tenu de l'isomorphisme (\ref{iso}), on obtient un morphisme
canonique de groupes abstraits
$$ \PP^{[1,\psi(2)]}(S)\rightarrow
\ker(J(S_{2})\rightarrow J(S_{1})).
$$
En particulier, le morphisme canonique de groupes
$\Pic^{\circ}(X)\rightarrow J(S_2)$ se factorise à travers
$\PP_{[\psi(2)]}(S)=\Pic^{\circ}(X_{\psi(2)})$:
$$
\xymatrix{\Pic^{\circ}(X)\ar[r]\ar[d]_{q}& \Pic^{\circ}(X_{\psi(2)})\ar[d]^{q_2}\\
J(S)\ar[r]& J(S_2)}.
$$
De plus, d'après \ref{cle}, le noyau du morphisme surjectif de
groupes $q_2:\Pic^{\circ}(X_{\psi(2)})\rightarrow J(S_2)$ est
engendré par $\mathcal{I}|_{X_{\psi(2)}}\in
\Pic^{\circ}(X_{\psi(2)})$.

\subsubsection{} Le cas général se fait par récurrence sur $n$, en
appliquant la même procédure (\ref{fidele}-\ref{preuve normal}).
Finalement, on obtient le théorème suivant:

\begin{theorem}\label{resultat final}Gardons les notations ci-dessus. Alors,
pour chaque $n\geq 1$, le morphisme de foncteurs
$q:\Pic^{\circ}_{X/S}\rightarrow J$ induit un morphisme de
$k$-schéma en groupes $q_n:\mathrm{Gr}({P}_{\psi(n)})\rightarrow
\mathrm{Gr}_n(J)$, qui est une isogénie. De plus,
$\ker(q_n)(k)=\{\mathcal{I}^{i}|_{X_{\psi(n)}}: ~i\in
\mathbf{Z}~\}\subset \Gr(\PP_{[\psi(n)]})(k)\simeq
\Pic^{\circ}(X_{\psi(n)})$.
\end{theorem}

\begin{remark} Posons $G^{[0]'}=G$, $\theta^{[0]'}=\theta:G\rightarrow J$.
Défnissons par récurrence, pour chaque entier $n\geq 0$, un
$S$-schéma en groupes lisse $G^{[n]'}$, et un morphisme fidèlement
plat de $S$-schémas en groupes $\theta^{[n]'}:G^{[n]'}\rightarrow
J^{[n]}$ de la manière suivante: supposons $n\geq 1$ et que l'on a
construit $G^{[n-1]'}$ et $\theta^{[n-1]'}:G^{[n-1]'}\rightarrow
J^{[n-1]}$. Puisque $\theta^{[n-1]'}$ est fidèlement plat, son
noyau $N^{[n-1]'}=\ker(\theta^{[n-1]'})$ est un $S$-schéma en
groupes plat sur $S$. Puis on définit $G^{[n]'}$ comme la
dilatation de $G^{[n-1]'}$ le long de
$\underline{N^{[n-1]'}}_{\mathrm{red}}\subset
\underline{G^{[n-1]'}}$. D'après \ref{dilatation}, $G^{[n]'}$ est
un $S$-schéma en groupes lisse. En vertu de la propriété
universelle des dilatations (\cite{BLR} 3.2/1), le morphisme
$\theta^{[n-1]'}$ induit un morphisme de $S$-schéma en groupes
$\theta^{[n]'}:G^{[n]'}\rightarrow J^{[n]}$, qui est aussi
fidèlement plat (\ref{pf rec} (iii)). Ceci finit la construction.
Notons $N^{[n]'}=\ker(\theta^{[n]'})$. Alors, la preuve du
théorème \ref{resultat final} (notamment \ref{pf rec} (ii)) montre
aussi que le schéma $N^{[n]'}$ est \emph{normal}. De plus, on peut
vérifier que le schéma $N^{[n]'}$ est lisse sur $S$ pour les
entiers $n$ assez grands.
\end{remark}

\begin{acknowledgements}Je suis très reconnaissant à M. Raynaud, qui m'a proposé ce
sujet. Il m'a donné aussi ses notes non-publiées
(\cite{Raynaud1}). Cet article a été préparé lors de mes séjours
comme post-doc à l'Université Duisburg-Essen. Je tiens à remercier
Hélène Esnault pour son hospitalité, et ses encouragements. Je
voudrais aussi exprimer ma reconnaissance à mes collegues de
Essen, pour les conversations intéressantes et pour leur aide dans
la vie quotidienne pendant mes séjours en Allemagne. Enfin,
j'aimerais remercier Wenqing, pour son soutien constant.
\end{acknowledgements}

\addcontentsline{toc}{section}{Bibliographie}

\begin{thebibliography}{99}



\bibitem{BLR} Bosch, S., L\"utkebohmert, W., Raynaud, M.,
\textit{Néron models}, Ergebnisse der Math., 3. Folge, 21. Berlin,
Springer 1990.



\bibitem{Greenberg} Greenberg, M. J., \textit{Schemata over local
rings}, Ann. Math., 73 (1961), p. 624-648.

\bibitem{Grothendieck} Grothendieck, A., \textit{Le groupe de Brauer
III}, dans: Dix exposés sur la cohomologie des schémas. North
Holland 1968

\bibitem{EGAIV} Grothendieck, A., \textit{Éléments de géométrie
algébrique (rédigés avec la collaboration de Jean Dieudonné: IV.
Etude locales des schémas et des morphismes de schémas, Quatrième
partie)}, Publications Mathématiques de l'I.H.É.S., 32 (1967), p.
5-361.

\bibitem{KatsuraUeno} Katsura, T., Ueno, K., \textit{On elliptic surfaces in
characteristic $p$}, Math. Ann., 272 (1985), p. 291-330.


\bibitem{Lipman} Lipman, J., \textit{The Picard group of a scheme over an Artin
ring}, Publications Mathématiques de l'I.H.É.S., 46 (1976), p.
15-86

\bibitem{LLR} Liu, Q., Lorenzini, D., Raynaud, M., \textit{Néron models, Lie algebras,
and reductions of curves of genus one}, Invent. Math., 157 (2004),
p. 455-518.


\bibitem{Mumford} Mumford, D., \textit{Enriques' classification of surfaces in char. $p$:
I}, in Global Analysis, Princeton Univ. Press, 1969.

\bibitem{Oort} Oort, F., \textit{Sur le schéma de Picard}, Bulletin de la S.M.F., 90 (1962), p.
1-14.

\bibitem{Raynaud} Raynaud, M., \textit{Spécialisation du foncteur de
Picard}, Publications Mathématiques de l'I.H.É.S., 38 (1970), p.
27-76.

\bibitem{Raynaud1} Raynaud, M., \textit{Surfaces elliptiqes et
quasi-elliptiques}, notes non publiées.

\bibitem{Serre} Serre, J.-P., \textit{Corps locaux}, Paris, Hermann
1968.

\bibitem{SerrePro} Serre, J.-P., \textit{Groupes pro-algébriques},
Publications Mathématiques de l'I.H.É.S., 7 (1960), p. 5-67.



\end{thebibliography}
\bibliographystyle{alpha}

 \end{document}